\documentclass{amsart}

\usepackage[section]{placeins}

\usepackage{amsmath}             
\usepackage{amscd}
\usepackage{amssymb}             
\usepackage{amsfonts}            
\usepackage{latexsym}            
\usepackage{amsthm}              
\usepackage{graphicx}

\usepackage{enumerate}
\usepackage{mathrsfs} 
\usepackage{thmtools}
\usepackage{thm-restate}

\usepackage{hyperref}
 
\newcommand{\Z}{\mathbb{Z}}         
   
\newcommand{\C}{\mathbb{C}}  

\newcommand{\Sp}{\operatorname{Sp}}

\newcommand{\GL}{\operatorname{GL}}

\newcommand{\SO}{\operatorname{SO}}

\newcommand{\rank}{\operatorname{rank}}
\newcommand{\soc}{\operatorname{soc}}
\newcommand{\Ext}{\operatorname{Ext}}

\newcommand{\Hom}{\operatorname{Hom}}

\newcommand{\Ker}{\operatorname{Ker}}

\newcommand{\chr}{\operatorname{char}}

\newcommand{\Ad}{\operatorname{Ad}}
\newcommand{\ad}{\operatorname{ad}}
\newcommand{\D}{\mathrm{d}}
\newcommand{\g}{\mathfrak{g}}
\newcommand{\diag}{\operatorname{diag}}

\newtheorem{lause}{Theorem}[section]

\newtheorem{lemma}[lause]{Lemma}
\newtheorem{seur}[lause]{Corollary}
\newtheorem{prop}[lause]{Proposition}
\newtheorem{prob}[lause]{Problem}

\newtheorem*{lause*}{Theorem}

\theoremstyle{definition}
\newtheorem{maar}[lause]{Definition}

\theoremstyle{remark}
\newtheorem{remark}[lause]{Remark}
\newtheorem{esim}[lause]{Example} 
\newtheorem*{mot*}{Motivation}
\newtheorem*{acknow*}{Acknowledgements}

\numberwithin{equation}{section}

\usepackage{xcolor}

\usepackage{hyphenat}
\allowdisplaybreaks

\begin{document}

\title[Adjoint Jordan blocks in characteristic two]{Adjoint Jordan blocks for simple algebraic groups of type $C_{\ell}$ in characteristic two}
\author{Mikko Korhonen}
\address{SUSTech International Center for Mathematics, Southern University of Science and Technology, \text{Shenzhen} 518055, Guangdong, P. R. China}
\email{korhonen\_mikko@hotmail.com}
\thanks{Support by Shenzhen Science and Technology Program (Grant No. RCBS20210609104420034).}
\dedicatory{In memory of Irina Suprunenko}
\date{\today}

\begin{abstract}
Let $G$ be a simple algebraic group over an algebraically closed field $K$ with Lie algebra $\g$. For unipotent elements $u \in G$ and nilpotent elements $e \in \g$, the Jordan block sizes of $\Ad(u)$ and $\ad(e)$ are known in most cases. In the cases that remain, the group $G$ is of classical type in bad characteristic, so $\chr K = 2$ and $G$ is of type $B_{\ell}$, $C_{\ell}$, or $D_{\ell}$. 

In this paper, we consider the case where $G$ is of type $C_{\ell}$ and $\chr K = 2$. As our main result, we determine the Jordan block sizes of $\Ad(u)$ and $\ad(e)$ for all unipotent $u \in G$ and nilpotent $e \in \g$. In the case where $G$ is of adjoint type, we will also describe the Jordan block sizes on $[\g, \g]$.\end{abstract}

\maketitle

\section{Introduction}

Let $G$ be a simple algebraic group over an algebraically closed field $K$, and denote the Lie algebra of $G$ by $\g$. Recall that an element $u \in G$ is \emph{unipotent}, if $f(u)$ is a unipotent linear map for every rational representation $f: G \rightarrow \GL(W)$. Similarly $e \in \g$ is said to be \emph{nilpotent}, if $\D f(e)$ is a nilpotent linear map for every rational representation $f: G \rightarrow \GL(W)$.

We denote the adjoint representations of $G$ and $\g$ by $\Ad: G \rightarrow \GL(\g)$ and $\ad: \g \rightarrow \mathfrak{gl}(V)$, respectively. In this paper, we will consider the following two problems.

	\begin{prob}\label{prob:unipAD}
		Let $u \in G$ be a unipotent element. What is the Jordan normal form of $\Ad(u)$?
	\end{prob}
	
	\begin{prob}\label{prob:nilAD}
		Let $e \in \g$ be a nilpotent element. What is the Jordan normal form of $\ad(e)$?
	\end{prob}

In most cases, the answer to both questions is known. When $G$ is simply connected of exceptional type, the Jordan block sizes were computed in the unipotent case by Lawther \cite{Lawther, LawtherCorrection} and in the nilpotent case by Stewart \cite{StewartNilpotentBlocks}. In the case where $G$ is exceptional of adjoint type, both questions are easily settled using the results of Lawther and Stewart, see \cite[Lemma 3.1]{KorhonenStewartThomas}.

When $G$ is of type $A_{\ell}$, solutions to both questions follow from the main results of \cite{KorhonenJordanGood} and \cite{KorhonenNilpotent}, see \cite[Remark 1.3]{KorhonenNilpotent}. For $G$ of type $B_{\ell}$, $C_{\ell}$, or $D_{\ell}$ in good characteristic, the Jordan block sizes are described by well-known results on decompositions of tensor products, symmetric squares, and exterior squares --- see for example \cite[Remark 3.5]{KorhonenNilpotent}. 

Thus it remains to solve Problem \ref{prob:unipAD} and Problem \ref{prob:nilAD} in the case where $G$ is of type $B_{\ell}$, $C_{\ell}$, or $D_{\ell}$ in characteristic $\chr K = 2$. In this paper, we consider the case where $G$ is of type $C_{\ell}$. We make the following assumption for the rest of this paper.
	\begin{center} \emph{Assume that $\chr K = 2$.}\end{center} 
	
Let $G$ be a simple algebraic group of type $C_{\ell}$ over $K$. As our main result, we determine the Jordan block sizes of $\Ad(u)$ and $\ad(e)$ for all unipotent $u \in G$ and nilpotent $e \in \g$. In type $C_{\ell}$ either $G$ is simply connected or $G$ is adjoint; we deal with these two possibilities as follows.

Let $G_{sc}$ be simply connected and simple of type $C_{\ell}$ with Lie algebra $\g_{sc}$, and let $G_{ad}$ be simple of adjoint type $C_{\ell}$ with Lie algebra $\g_{ad}$. There exists an isogeny $\varphi: G_{sc} \rightarrow G_{ad}$, which induces a bijection between the unipotent variety of $G_{sc}$ and $G_{ad}$. Similarly the differential $\D \varphi: \g_{sc} \rightarrow \g_{ad}$ induces a bijection between the nilpotent cone of $\g_{sc}$ and $\g_{ad}$.

Thus for the solution of Problem \ref{prob:unipAD} and Problem \ref{prob:nilAD} in type $C_{\ell}$, it will suffice to consider the Jordan block sizes of unipotent $u \in G_{sc}$ on $\g_{sc}$ and $\g_{ad}$, and the Jordan block sizes of nilpotent $e \in \g_{sc}$ on $\g_{sc}$ and $\g_{ad}$.

We can assume that $G_{sc} = \Sp(V)$ with Lie algebra $\g_{sc} = \mathfrak{sp}(V)$, where $\dim V = 2\ell$. It is well known (Lemma \ref{lemma:LieSpS2}) that $\g_{sc} \cong S^2(V)^*$ as $G_{sc}$-modules. The Jordan block sizes of the action of unipotent $u \in \GL(V)$ on $S^2(V)$ and nilpotent $e \in \mathfrak{gl}(V)$ on $S^2(V)$ are described in \cite{KorhonenSymExt2021}. Taking the dual does not change the Jordan block sizes, so the Jordan block sizes on $\g_{sc}$ are known by previous results.

In our main results, we will describe the Jordan block sizes on $\g_{ad}$ in terms of the Jordan block sizes on $\g_{sc}$. To state the result in the unipotent case, we first need describe the classification of unipotent conjugacy classes in $\Sp(V)$, which in characteristic two is due to Hesselink \cite{Hesselink}. We discuss this in some more detail in Section \ref{section:hesselinkunipotent}.

Let $C_q$ be a cyclic group of order $q$, where $q = 2^{\alpha}$ for some $\alpha \geq 0$. Then there are a total of $q$ indecomposable $K[C_q]$-modules $V_1$, $\ldots$, $V_q$ up to isomorphism, where $\dim V_i = i$ and a generator of $C_q$ acts on $V_i$ with a single $i \times i$ unipotent Jordan block. For convenience we will denote $V_0 = 0$. If $W$ is a vector space over $K$, we denote $W^0 = 0$, and $W^d = W \oplus \cdots \oplus W$ ($d$ summands) for an integer $d > 0$.

Suppose then that $u \in \GL(V)$ is unipotent, so $u$ has order $q$ for some $q = 2^{\alpha}$. We denote the group algebra of $\langle u \rangle$ by $K[u]$. Then $V \downarrow K[u] \cong V_{d_1}^{n_1} \oplus \cdots \oplus V_{d_t}^{n_t}$ for some integers $0 < d_1 < \cdots < d_t$, where $n_i > 0$ for all $1 \leq i \leq t$. Equivalently, the Jordan normal form of $u$ has Jordan blocks of sizes $d_1$, $\ldots$, $d_t$, and a block of size $d_i$ has multiplicity $n_i$.

For unipotent $u \in \Sp(V)$, we have a decomposition $V \downarrow K[u] = U_1 \perp \cdots \perp U_t$, where $U_i$ are \emph{orthogonally indecomposable} $K[u]$-modules. Here orthogonally indecomposable means that if $U_i = U' \perp U''$ as $K[u]$-modules, then $U' = 0$ or $U'' = 0$. The orthogonally indecomposable $K[u]$-modules fall into two types: one denoted by $V(m)$ (for $m$ even) and another denoted by $W(m)$ (for $m \geq 1$). We define these in Section \ref{section:hesselinkunipotent}, for now we only mention that $V(m) \cong V_m$ and $W(m) \cong V_m \oplus V_m$ as $K[u]$-modules.

Our first main result is the following. (Below we denote by $\nu_2$ the $2$-adic valuation on the integers, so $\nu_2(a)$ is the largest integer $k \geq 0$ such that $2^k$ divides $a$.)

\begin{lause}\label{thm:unipGSCtoGAD}
Let $u \in \Sp(V)$ be unipotent, with orthogonal decomposition $$V \downarrow K[u] = \sum_{1 \leq i \leq t} W(m_i) \perp \sum_{1 \leq j \leq s} V(2k_j).$$ Denote by $\alpha \geq 0$ the largest integer such that $2^{\alpha} \mid m_i,k_j$ for all $i$ and $j$. Then the following hold:
	\begin{enumerate}[\normalfont (i)]
		\item Suppose that $s = 0$. Then:
			\begin{enumerate}[\normalfont (a)]
				\item If $\alpha = 0$, then $\g_{sc} \cong \g_{ad}$ as $K[u]$-modules.
				\item If $\alpha > 0$, then \begin{align*}
		\g_{sc} &\cong V_{2^{\alpha}} \oplus V' \\
		\g_{ad} &\cong V_{1} \oplus V_{2^{\alpha}-1} \oplus V' \end{align*} for some $K[u]$-module $V'$.
			\end{enumerate}
		\item Suppose that $s > 0$, and let $\beta = \max_{1 \leq j \leq s} \nu_2(k_j)$. Then:
			\begin{enumerate}[\normalfont (a)]
				\item If $\beta = \nu_2(k_j)$ for all $1 \leq j \leq s$ and $\nu_2(m_i) > \beta$ for all $1 \leq i \leq t$, then $\g_{sc} \cong \g_{ad}$ as $K[u]$-modules.
				\item If $\beta > \nu_2(k_j)$ for some $1 \leq j \leq s$, or $\nu_2(m_i) \leq \beta$ for some $1 \leq i \leq t$, then  \begin{align*}
		\g_{sc} &\cong V_{2^{\alpha}} \oplus V_{2^{\beta}} \oplus V' \\
		\g_{ad} &\cong V_{2^{\alpha}-1} \oplus V_{2^{\beta}+1} \oplus V' \end{align*} for some $K[u]$-module $V'$.		
			\end{enumerate}
	\end{enumerate}
\end{lause}

To describe our results in the nilpotent case, we recall the classification of nilpotent orbits in $\mathfrak{sp}(V)$, due to Hesselink \cite{Hesselink}. For more details we refer to Section \ref{section:hesselinknilpotent}.

Suppose that $q = 2^{\alpha}$ for some $\alpha \geq 0$. Let $\mathfrak{w}_q$ be the abelian $2$-Lie algebra over $K$ generated by a single nilpotent element $e \in \mathfrak{w}_q$ such that $e^{[2^{\alpha}]} = 0$ and $e^{[2^{\alpha-1}]} \neq 0$. There are a total of $q$ indecomposable $\mathfrak{w}_q$-modules $W_1$, $\ldots$, $W_q$ up to isomorphism, where $\dim W_i = i$ and $e$ acts on $W_i$ with a single $i \times i$ nilpotent Jordan block. Throughout we will denote $W_0 = 0$.

Consider then a nilpotent linear map $e \in \mathfrak{gl}(V)$, and denote the $2$-Lie subalgebra generated by $e$ with $K[e]$. Then $K[e] \cong \mathfrak{w}_q$, where $q$ is the smallest power of two such that $e^q = 0$. We have $V \downarrow K[e] \cong W_{d_1}^{n_1} \oplus \cdots \oplus W_{d_t}^{n_t}$ for some integers $0 < d_1 < \cdots < d_t$, where $n_i > 0$ for all $1 \leq i \leq t$. As in the unipotent case, this amounts to the statement that the Jordan normal form of $e$ has Jordan blocks of sizes $d_1$, $\ldots$, $d_t$, and a block of size $d_i$ has multiplicity $n_i$.

For nilpotent $e \in \mathfrak{sp}(V)$, we have an orthogonal decomposition $V \downarrow K[e] = U_1 \perp \cdots \perp U_t$, where the $U_i$ are orthogonally indecomposable $K[e]$-modules. Here orthogonally indecomposable is defined as in the unipotent case. For the orthogonally indecomposable $K[e]$-modules, there are several different types: $V(m)$ (for $m$ even), $W_k(m)$ (for $0 < k < m/2$ and $m > 2$), and $W(m)$ (for $m \geq 1$). We give the definitions and more details in Section \ref{section:hesselinknilpotent}, for now we just note that $V(m) \cong W_m$, $W_k(m) \cong W_m \oplus W_m$, and $W(m) \cong W_m \oplus W_m$ as $K[e]$-modules.

Our main result in the nilpotent case is the following.

\begin{lause}\label{thm:nilGSCtoGAD}
Let $e \in \mathfrak{sp}(V)$ be nilpotent, with orthogonal decomposition $$V \downarrow K[e] = \sum_{1 \leq i \leq t} W(m_i) \perp \sum_{1 \leq j \leq t'} W_{k_j}(\ell_j) \perp \sum_{1 \leq r \leq t''} V(2d_r).$$ Denote by $\alpha \geq 0$ the largest integer such that $2^{\alpha} \mid m_i,\ell_j,2d_r$ for all $i$, $j$, and $r$. Then the following statements hold:
	\begin{enumerate}[\normalfont (i)]
		\item Suppose that $V \downarrow K[e] = \sum_{1 \leq i \leq t} W(m_i)$. Then:	
			\begin{enumerate}[\normalfont (a)]
				\item If $\alpha = 0$, then $\g_{sc} \cong \g_{ad}$ as $K[e]$-modules.
				\item If $\alpha > 0$, then \begin{align*}
		\g_{sc} &\cong W_{2^{\alpha}} \oplus V' \\
		\g_{ad} &\cong W_{1} \oplus W_{2^{\alpha}-1} \oplus V' \end{align*} for some $K[e]$-module $V'$.
			\end{enumerate}
		\item Suppose that $V \downarrow K[e]$ is not of the form $\sum_{1 \leq i \leq t} W(m_i)$. Then:
			\begin{enumerate}[\normalfont (a)]
				\item If $\alpha = 1$, then $\g_{sc} \cong \g_{ad}$ as $K[e]$-modules.
				\item If $\alpha \neq 1$, then \begin{align*}
		\g_{sc} &\cong W_1 \oplus W_{2^{\alpha}} \oplus V' \\
		\g_{ad} &\cong W_{2} \oplus W_{2^{\alpha}-1} \oplus V' \end{align*} for some $K[e]$-module $V'$.
			\end{enumerate}
	\end{enumerate}
\end{lause}

\begin{esim}\label{example:tables}
In Table \ref{table:UNIPexamples} and Table \ref{table:NILexamples}, we illustrate Theorem \ref{thm:unipGSCtoGAD} and Theorem \ref{thm:nilGSCtoGAD} in the case where $G$ is of type $C_{\ell}$ for $2 \leq \ell \leq 4$. In the tables we have also included the Jordan block sizes on $[\g_{ad},\g_{ad}]$, which we prove as an intermediate result in Proposition \ref{prop:unipKeraction} and Proposition \ref{prop:nilKeraction}. Furthermore, the values of $\alpha$ and $\beta$ that appear in Theorem \ref{thm:unipGSCtoGAD} and Theorem \ref{thm:nilGSCtoGAD} are included.

In the tables, we use the notation $$d_1^{n_1}, \ldots, d_t^{n_t}$$ for $0 < d_1 < \cdots < d_t$ and $n_i > 0$ to denote that the Jordan block sizes are $d_1$, $\ldots$, $d_t$, and a block of size $d_i$ has multiplicity $n_i$. For the orthogonal decompositions in the first column, notation such as $V(m)^k$ denotes an orthogonal direct sum $V(m) \perp \cdots \perp V(m)$ ($k$ summands).\end{esim}

In particular, from Theorem \ref{thm:unipGSCtoGAD} and Theorem \ref{thm:nilGSCtoGAD} we get the number of Jordan blocks of $\Ad(u)$ and $\ad(e)$, which is equal to the dimension of the Lie algebra centralizer of $u$ and $e$, respectively.

\begin{seur}\label{cor:centralizerdimUNIP}
Let $u \in \Sp(V)$ be unipotent and denote $\alpha$, $s$ as in Theorem \ref{thm:unipGSCtoGAD}. Then the following hold:
	\begin{enumerate}[\normalfont (i)]
		\item Suppose that $s = 0$. Then $$\dim \g_{ad}^u = \begin{cases} \dim \g_{sc}^u,& \text{ if } \alpha = 0. \\ \dim \g_{sc}^u+1,& \text{ if } \alpha > 0.\end{cases}$$		
		\item Suppose that $s > 0$. Then $$\dim \g_{ad}^u = \begin{cases} 
		\dim \g_{sc}^u-1,& \text{ if } \alpha = 0 \text{ and } \nu_2(k_j) > 0 \text{ for some } j.  \\ 
		\dim \g_{sc}^u-1,& \text{ if } \alpha = 0 \text{ and } \nu_2(m_i) = 0 \text{ for some } i. \\ 
		\dim \g_{sc}^u,& \text{ if } \alpha = 0 \text{ and } \nu_2(k_j) = 0, \nu_2(m_i) > 0 \text{ for all } i \text{ and } j. \\
		\dim \g_{sc}^u,& \text{ if } \alpha > 0.\end{cases}$$			
	\end{enumerate}
\end{seur}

\begin{seur}\label{cor:centralizerdimNIL}
Let $e \in \mathfrak{sp}(V)$ be nilpotent and denote $\alpha$ as in Theorem \ref{thm:unipGSCtoGAD}. Then the following hold:
	\begin{enumerate}[\normalfont (i)]
		\item Suppose that $V \downarrow K[e] = \sum_{1 \leq i \leq t} W(m_i)$. Then $$\dim \g_{ad}^e = \begin{cases} \dim \g_{sc}^e,& \text{ if } \alpha = 0. \\ \dim \g_{sc}^e+1,& \text{ if } \alpha > 0.\end{cases}$$
		\item Suppose that $V \downarrow K[e]$ is not of the form $\sum_{1 \leq i \leq t} W(m_i)$. Then $$\dim \g_{ad}^e = \begin{cases} \dim \g_{sc}^e-1,& \text{ if } \alpha = 0 .\\ \dim \g_{sc}^e,& \text{ if } \alpha > 0.\end{cases}$$
	\end{enumerate}
\end{seur}

In the case of regular elements, this amounts to the following. (An element $x \in G$ or $x \in \g$ is \emph{regular}, if $\dim C_G(x) = \rank G$. In the case of $G = \Sp(V)$, a regular unipotent element is characterized by $V \downarrow K[u] = V(2\ell)$, and a regular nilpotent element is characterized by $V \downarrow K[e] = V(2\ell)$.)

\begin{seur}\label{cor:regularcentralizerdim}
Suppose that $G$ is of type $C_{\ell}$ in characteristic two (adjoint or simply connected). Let $u \in G$ be a regular unipotent element and $e \in \g$ a regular nilpotent element. Then $\dim \g^u = \ell+1$ and $\dim \g^e = 2\ell$.
\end{seur}

For the proofs of our main results, our basic approach is as follows. We will first describe the Jordan block sizes of unipotent $u \in G_{sc}$ and nilpotent $e \in \g_{sc}$ on $\g_{sc}/Z(\g_{sc})$, in terms of the Jordan block sizes on $\g_{sc}$. This is attained in Section \ref{section:unipgmodZ} and Section \ref{section:nilgmodZ}.

It is known that $\g_{sc}/Z(\g_{sc}) \cong [\g_{ad}, \g_{ad}]$ as $G_{sc}$-modules (Lemma \ref{lemma:gmodZisgadgad}), so we then have the Jordan block sizes on $[\g_{ad}, \g_{ad}]$ as well. In Section \ref{section:unipGAdZ} and Section \ref{section:nilGAdZ} we will describe the Jordan block sizes of unipotent $u \in G_{sc}$ and nilpotent $e \in \g_{sc}$ on $\g_{ad}$, in terms of the Jordan block sizes on $[\g_{ad}, \g_{ad}]$. Combining these results, we get the Jordan block sizes of $u$ and $e$ on $\g_{ad}$, in terms of the Jordan block sizes on $\g_{sc}$ (Theorem \ref{thm:unipGSCtoGAD}, Theorem \ref{thm:nilGSCtoGAD}).

The other sections of this paper are organized as follows. The notation, terminology, and some preliminary results are stated in Section \ref{section:notation} and Section \ref{section:prelim}. In Section \ref{section:tensor}, we state results on Jordan block sizes of unipotent and nilpotent elements on tensor products, symmetric squares, and exterior squares. The classification of unipotent classes in $\Sp(V)$ and nilpotent orbits in $\mathfrak{sp}(V)$ is described in Section \ref{section:hesselinkunipotent} and Section \ref{section:hesselinknilpotent}, respectively.

In Section \ref{section:chevalley}, we discuss the Chevalley construction of simple algebraic groups, and in particular the action of $G_{sc} = \Sp(V)$ on $\g_{ad}$. We then make some observations about the structure of $\g_{sc}$ and $\g_{ad}$ as a $G_{sc}$-module in Section \ref{section:typeClie}. 

As mentioned earlier, in Section \ref{section:unipgmodZ} and Section \ref{section:nilgmodZ} we describe the Jordan block sizes of unipotent and nilpotent elements on $\g_{sc}/Z(\g_{sc})$, in terms of the Jordan block sizes on $\g_{sc}$. In Section \ref{section:unipGAdZ} and Section \ref{section:nilGAdZ} we similarly describe the Jordan block sizes on $[\g_{ad}, \g_{ad}]$, in terms of Jordan block sizes on $\g_{ad}$. These results allow us to prove our main results, and the proofs of the results stated in this introduction are given in Section \ref{section:final}.

\begin{table}[!htbp]
\centering

\caption{For $G_{sc} = \Sp(V)$ simply connected of type $C_{\ell}$ with $2 \leq \ell \leq 4$ and $\chr K = 2$, Jordan block sizes of unipotent $u \in G_{sc}$ on $\g_{sc}$, $[\g_{ad}, \g_{ad}]$, and $\g_{ad}$. See Example \ref{example:tables}.}\label{table:UNIPexamples}
\begin{tabular}{lllllll}
\hline

&&&&&&\\[-10pt]
$\ell$ & $V \downarrow K[u]$          & $\g_{sc} \downarrow K[u]$ & $[\g_{ad}, \g_{ad}] \downarrow K[u]$ & $\g_{ad} \downarrow K[u]$ & $\alpha$ & $\beta$ \\ \hline
&&&&&&\\[-5pt]
$2$   & $V(4)$ &        $ 2, 4^{2} $ &         $ 1, 4^{2} $ &         $ 2, 4^{2} $      & $1$ & $1$ \\
& $V(2)^2$ &      $ 1^{2}, 2^{4} $ &     $ 1, 2^{4} $ &     $ 1^{2}, 2^{4} $            & $0$ & $0$ \\
& $W(1) \perp V(2)$ & $ 1^{4}, 2^{3} $ &     $ 1^{3}, 2^{3} $ &     $ 1^{2}, 2^{4} $    & $0$ & $0$ \\
& $W(2)$ &        $ 1^{2}, 2^{4} $ &     $ 1^{3}, 2^{3} $ &     $ 1^{4}, 2^{3} $        & $1$ & $-$ \\
& $W(1)^2$ &      $ 1^{10} $ &           $ 1^{9} $ &           $ 1^{10} $               & $0$ & $-$ \\
\\
&&&&&&\\		
$3$	   & $V(6)$ &          $ 1, 4, 8^{2} $ &      $ 4, 8^{2} $ &      $ 1, 4, 8^{2} $      & $0$ & $0$ \\
& $V(2) \perp V(4)$ &   $ 1, 2^{2}, 4^{4} $ &  $ 2^{2}, 4^{4} $ &  $ 2, 3, 4^{4} $         & $0$ & $1$ \\
& $V(2)^3$ &        $ 1^{3}, 2^{9} $ &     $ 1^{2}, 2^{9} $ &     $ 1^{3}, 2^{9} $         & $0$ & $0$ \\
& $W(1)^2 \perp V(2)$ & $ 1^{11}, 2^{5} $ &    $ 1^{10}, 2^{5} $ &    $ 1^{9}, 2^{6} $     & $0$ & $0$ \\
& $W(1) \perp V(4)$ &   $ 1^{3}, 2, 4^{4} $ &  $ 1^{2}, 2, 4^{4} $ &  $ 1^{2}, 3, 4^{4} $  & $0$ & $1$ \\
& $W(1) \perp V(2)^2$ & $ 1^{5}, 2^{8} $ &     $ 1^{4}, 2^{8} $ &     $ 1^{3}, 2^{9} $     & $0$ & $0$ \\
& $W(3)$ &          $ 1, 2^{2}, 4^{4} $ &  $ 2^{2}, 4^{4} $ &  $ 1, 2^{2}, 4^{4} $         & $0$ & $-$ \\
& $W(1) \perp W(2)$ &   $ 1^{5}, 2^{8} $ &     $ 1^{4}, 2^{8} $ &     $ 1^{5}, 2^{8} $     & $0$ & $-$ \\
& $W(1)^3$ &        $ 1^{21} $ &           $ 1^{20} $ &           $ 1^{21} $               & $0$ & $-$ \\
&&&&&&\\
$4$  & $V(8)$ &               $ 4, 8^{4} $ &         $ 3, 8^{4} $ &         $ 4, 8^{4} $                                      & $2$ & $2$ \\
& $V(2) \perp V(6)$ &        $ 1^{2}, 2, 4, 6^{2}, 8^{2} $ & $ 1, 2, 4, 6^{2}, 8^{2} $ & $ 1^{2}, 2, 4, 6^{2}, 8^{2} $        & $0$ & $0$ \\
& $V(4)^2$ &             $ 2^{2}, 4^{8} $ &     $ 1, 2, 4^{8} $ &     $ 2^{2}, 4^{8} $                                        & $1$ & $1$ \\
& $V(2)^2 \perp V(4)$ &      $ 1^{2}, 2^{5}, 4^{6} $ & $ 1, 2^{5}, 4^{6} $ & $ 1, 2^{4}, 3, 4^{6} $                           & $0$ & $1$ \\
& $V(2) \perp W(3)$ &        $ 1^{2}, 2^{5}, 4^{6} $ & $ 1, 2^{5}, 4^{6} $ & $ 2^{6}, 4^{6} $                                 & $0$ & $0$ \\
& $W(1) \perp V(2)^3$ &      $ 1^{6}, 2^{15} $ &    $ 1^{5}, 2^{15} $ &    $ 1^{4}, 2^{16} $                                  & $0$ & $0$ \\
& $W(1)^3 \perp V(2)$ &      $ 1^{22}, 2^{7} $ &    $ 1^{21}, 2^{7} $ &    $ 1^{20}, 2^{8} $                                  & $0$ & $0$ \\
& $W(2) \perp V(4)$ &        $ 1^{2}, 2^{5}, 4^{6} $ & $ 1^{3}, 2^{4}, 4^{6} $ & $ 1^{3}, 2^{3}, 3, 4^{6} $                   & $1$ & $1$ \\
& $V(2)^4$ &             $ 1^{4}, 2^{16} $ &    $ 1^{3}, 2^{16} $ &    $ 1^{4}, 2^{16} $                                      & $0$ & $0$ \\
& $W(1)^2 \perp V(4)$ &      $ 1^{10}, 2, 4^{6} $ & $ 1^{9}, 2, 4^{6} $ & $ 1^{9}, 3, 4^{6} $                                 & $0$ & $1$ \\
& $W(1)^2 \perp V(2)^2$ &    $ 1^{12}, 2^{12} $ &   $ 1^{11}, 2^{12} $ &   $ 1^{10}, 2^{13} $                                 & $0$ & $0$ \\
& $W(1) \perp V(6)$ &        $ 1^{4}, 4, 6^{2}, 8^{2} $ & $ 1^{3}, 4, 6^{2}, 8^{2} $ & $ 1^{2}, 2, 4, 6^{2}, 8^{2} $          & $0$ & $0$ \\
& $W(1) \perp V(2) \perp V(4)$ & $ 1^{4}, 2^{4}, 4^{6} $ & $ 1^{3}, 2^{4}, 4^{6} $ & $ 1^{3}, 2^{3}, 3, 4^{6} $               & $0$ & $1$ \\
& $W(4)$ &               $ 2^{2}, 4^{8} $ &     $ 2^{2}, 3, 4^{7} $ &     $ 1, 2^{2}, 3, 4^{7} $                              & $2$ & $-$ \\
& $W(1) \perp W(3)$ &        $ 1^{4}, 2^{2}, 3^{4}, 4^{4} $ & $ 1^{3}, 2^{2}, 3^{4}, 4^{4} $ & $ 1^{4}, 2^{2}, 3^{4}, 4^{4} $ & $0$ & $-$ \\
& $W(2)^2$ &             $ 1^{4}, 2^{16} $ &    $ 1^{5}, 2^{15} $ &    $ 1^{6}, 2^{15} $                                      & $1$ & $-$ \\
& $W(1)^2 \perp W(2)$ &      $ 1^{12}, 2^{12} $ &   $ 1^{11}, 2^{12} $ &   $ 1^{12}, 2^{12} $                                 & $0$ & $-$ \\
& $W(1)^4$ &             $ 1^{36} $ &           $ 1^{35} $ &           $ 1^{36} $                                             & $0$ & $-$ \\
&&&&&&\\ \hline
\end{tabular}
\end{table}

\begin{table}[!htbp]
\centering
\caption{For $G_{sc} = \Sp(V)$ simply connected of type $C_{\ell}$ with $2 \leq \ell \leq 4$ and $\chr K = 2$, Jordan block sizes of nilpotent $e \in \g_{sc}$ on $\g_{sc}$, $[\g_{ad}, \g_{ad}]$, and $\g_{ad}$. See Example \ref{example:tables}.}\label{table:NILexamples}
\begin{tabular}{llllll}
\hline
&&&&&\\[-10pt]
$\ell$ & $V \downarrow K[e]$          & $\g_{sc} \downarrow K[e]$ & $[\g_{ad}, \g_{ad}] \downarrow K[e]$ & $\g_{ad} \downarrow K[e]$ & $\alpha$ \\ \hline
&&&&&\\[-5pt]
$2$    & $V(4)$ &        $ 1^{2}, 4^{2} $ &     $ 1^{2}, 3, 4 $ &     $ 1, 2, 3, 4 $          & $2$ \\
       & $V(2)^2$ &      $ 1^{2}, 2^{4} $ &     $ 1^{3}, 2^{3} $ &     $ 1^{2}, 2^{4} $       & $1$ \\
       & $W(1)  \perp  V(2)$ & $ 1^{4}, 2^{3} $ &     $ 1^{3}, 2^{3} $ &     $ 1^{2}, 2^{4} $ & $0$ \\
       & $W(2)$ &        $ 1^{2}, 2^{4} $ &     $ 1^{3}, 2^{3} $ &     $ 1^{4}, 2^{3} $       & $1$ \\
       & $W(1)^2$ &      $ 1^{10} $ &           $ 1^{9} $ &           $ 1^{10} $              & $0$ \\
&&&&&\\	
	
$3$	   & $V(6)$ &        $ 1^{3}, 2, 8^{2} $ &  $ 1^{4}, 8^{2} $ &     $ 1^{3}, 2, 8^{2} $    & $1$ \\
& $W_1(3)$ &      $ 1^{5}, 4^{4} $ &     $ 1^{4}, 4^{4} $ &     $ 1^{3}, 2, 4^{4} $           & $0$ \\
& $V(2) \perp V(4)$ &   $ 1^{3}, 2, 4^{4} $ &  $ 1^{4}, 4^{4} $ &     $ 1^{3}, 2, 4^{4} $     & $1$ \\
& $W(1)^2 \perp V(2)$ & $ 1^{11}, 2^{5} $ &    $ 1^{10}, 2^{5} $ &    $ 1^{9}, 2^{6} $        & $0$ \\
& $W(1) \perp V(2)^2$ & $ 1^{5}, 2^{8} $ &     $ 1^{4}, 2^{8} $ &     $ 1^{3}, 2^{9} $        & $0$ \\
& $W(1) \perp V(4)$ &   $ 1^{5}, 4^{4} $ &     $ 1^{4}, 4^{4} $ &     $ 1^{3}, 2, 4^{4} $     & $0$ \\
& $V(2)^3$ &      $ 1^{3}, 2^{9} $ &     $ 1^{4}, 2^{8} $ &     $ 1^{3}, 2^{9} $              & $1$ \\
& $W(1) \perp W(2)$ &   $ 1^{5}, 2^{8} $ &     $ 1^{4}, 2^{8} $ &     $ 1^{5}, 2^{8} $        & $0$ \\
& $W(3)$ &        $ 1^{5}, 4^{4} $ &     $ 1^{4}, 4^{4} $ &     $ 1^{5}, 4^{4} $              & $0$ \\
& $W(1)^3$ &      $ 1^{21} $ &           $ 1^{20} $ &           $ 1^{21} $                    & $0$ \\
&&&&&\\

$4$  & $V(8)$ &           $ 1^{4}, 8^{4} $ &     $ 1^{4}, 7, 8^{3} $ &     $ 1^{3}, 2, 7, 8^{3} $                             & $3$ \\
     & $W_1(4)$ &         $ 1^{4}, 4^{8} $ &     $ 1^{4}, 3, 4^{7} $ &     $ 1^{3}, 2, 3, 4^{7} $                             & $2$ \\
     & $W(1) \perp W_1(3)$ &    $ 1^{6}, 2^{3}, 4^{6} $ & $ 1^{5}, 2^{3}, 4^{6} $ & $ 1^{4}, 2^{4}, 4^{6} $                   & $0$ \\
     & $V(2) \perp V(6)$ &      $ 1^{4}, 2^{2}, 6^{2}, 8^{2} $ & $ 1^{5}, 2, 6^{2}, 8^{2} $ & $ 1^{4}, 2^{2}, 6^{2}, 8^{2} $  & $1$ \\
     & $V(4)^2$ &         $ 1^{4}, 4^{8} $ &     $ 1^{4}, 3, 4^{7} $ &     $ 1^{3}, 2, 3, 4^{7} $                             & $2$ \\
     & $V(2)^2 \perp V(4)$ &    $ 1^{4}, 2^{4}, 4^{6} $ & $ 1^{5}, 2^{3}, 4^{6} $ & $ 1^{4}, 2^{4}, 4^{6} $                   & $1$ \\
     & $W(1) \perp V(6)$ &      $ 1^{6}, 2, 6^{2}, 8^{2} $ & $ 1^{5}, 2, 6^{2}, 8^{2} $ & $ 1^{4}, 2^{2}, 6^{2}, 8^{2} $      & $0$ \\
     & $W(1) \perp V(2) \perp V(4)$ & $ 1^{6}, 2^{3}, 4^{6} $ & $ 1^{5}, 2^{3}, 4^{6} $ & $ 1^{4}, 2^{4}, 4^{6} $             & $0$ \\
     & $W(1)^2 \perp V(4)$ &    $ 1^{12}, 4^{6} $ &    $ 1^{11}, 4^{6} $ &    $ 1^{10}, 2, 4^{6} $                            & $0$ \\
     & $W(1)^2 \perp V(2)^2$ &  $ 1^{12}, 2^{12} $ &   $ 1^{11}, 2^{12} $ &   $ 1^{10}, 2^{13} $                              & $0$ \\
     & $W(1)^3 \perp V(2)$ &    $ 1^{22}, 2^{7} $ &    $ 1^{21}, 2^{7} $ &    $ 1^{20}, 2^{8} $                               & $0$ \\
     & $W(2) \perp V(4)$ &      $ 1^{4}, 2^{4}, 4^{6} $ & $ 1^{5}, 2^{3}, 4^{6} $ & $ 1^{4}, 2^{4}, 4^{6} $                   & $1$ \\
     & $V(2)^4$ &         $ 1^{4}, 2^{16} $ &    $ 1^{5}, 2^{15} $ &    $ 1^{4}, 2^{16} $                                     & $1$ \\
     & $W(1) \perp V(2)^3$ &    $ 1^{6}, 2^{15} $ &    $ 1^{5}, 2^{15} $ &    $ 1^{4}, 2^{16} $                               & $0$ \\
		 & $V(2) \perp W(3)$ &      $ 1^{6}, 2^{3}, 4^{6} $ & $ 1^{5}, 2^{3}, 4^{6} $ & $ 1^{4}, 2^{4}, 4^{6} $                   & $0$ \\
     & $W(1)^2 \perp W(2)$ &    $ 1^{12}, 2^{12} $ &   $ 1^{11}, 2^{12} $ &   $ 1^{12}, 2^{12} $                              & $0$ \\
     & $W(2)^2$ &         $ 1^{4}, 2^{16} $ &    $ 1^{5}, 2^{15} $ &    $ 1^{6}, 2^{15} $                                     & $1$ \\
     & $W(1) \perp W(3)$ &      $ 1^{8}, 3^{4}, 4^{4} $ & $ 1^{7}, 3^{4}, 4^{4} $ & $ 1^{8}, 3^{4}, 4^{4} $                   & $0$ \\
     & $W(4)$ &           $ 1^{4}, 4^{8} $ &     $ 1^{4}, 3, 4^{7} $ &     $ 1^{5}, 3, 4^{7} $                                & $2$ \\
     & $W(1)^4$ &         $ 1^{36} $ &           $ 1^{35} $ &           $ 1^{36} $                                            & $0$ \\
&&&&&\\ \hline
\end{tabular}
\end{table}

\section{Notation and terminology}\label{section:notation}

 We will use the following notation and terminology throughout the paper.

\subsection{Generalities} We will always assume that $K$ is an algebraically closed field of characteristic two.

For a $K$-vector space $V$ and integer $n > 0$, we denote $V^n = V \oplus \cdots \oplus V$, where $V$ occurs $n$ times. Furthermore, we denote $V^0 = 0$.

We will always use $G$ to denote a group, and all the $K[G]$-modules that we consider will be finite-dimensional. Suppose that $V = U_1 \supset U_2 \supset \cdots \supset U_t \supset U_{t+1} = 0$ is the socle filtration of a $K[G]$-module $V$, so $Z_i := U_i/U_{i+1} = \soc(V/U_{i+1})$ for all $1 \leq i \leq t$. Then we will denote this by $V = Z_1|Z_2|\cdots|Z_t$, and $V$ is said to be \emph{uniserial} if $Z_i$ is irreducible for all $1 \leq i \leq t$.

We denote by $\nu_2$ the $2$-adic valuation on the integers, so $\nu_2(a)$ is the largest integer $k \geq 0$ such that $2^k$ divides $a$.

\subsection{Algebraic groups} Suppose that $G$ is a simple linear algebraic group over $K$. In the context of algebraic groups, the notation that we use will be as in \cite{JantzenBook}. We will denote the Lie algebra of $G$ by $\g$.

When $G$ is an algebraic group, by a $G$-module we will always mean a finite-dimensional rational $K[G]$-module. Fix a maximal torus $T$ of $G$ with character group $X(T)$, and a base $\Delta = \{ \alpha_1, \ldots, \alpha_{\ell} \}$ for the root system of $G$, where $\ell = \operatorname{rank} G$. We will always use the standard Bourbaki labeling of the simple roots $\alpha_i$, as given in \cite[11.4, p. 58]{Humphreys}. We denote the dominant weights with respect to $\Delta$ by $X(T)^+$, and the fundamental dominant weight corresponding to $\alpha_i$ is denoted by $\varpi_i$. For a dominant weight $\lambda \in X(T)^+$, we denote the rational irreducible $K[G]$-module with highest weight $\lambda$ by $L_G(\lambda)$.

An element $u \in G$ is \emph{unipotent}, if $f(u)$ is a unipotent linear map for every rational representation $f: G \rightarrow \GL(W)$. Since $\chr K = 2$, equivalently $u$ is unipotent if and only if $u$ has order $2^{\alpha}$ for some $\alpha \geq 0$.

Similarly $e \in \g$ is said to be \emph{nilpotent}, if $\D f(e)$ is a nilpotent linear map for every rational representation $f: G \rightarrow \GL(W)$. Equivalently $e$ is nilpotent if and only if $e^{[2^\alpha]} = 0$ for some $\alpha > 0$, where $x \mapsto x^{[2]}$ is the canonical $p$-mapping on $\g$.

We will mostly be concerned with the case where $G = \Sp(V)$, in which case $G$ is of type $C_{\ell}$, where $\dim V = 2\ell$. In this case $u \in \Sp(V)$ is unipotent if and only if $u$ is a unipotent linear map on $V$. Furthermore $\g = \mathfrak{sp}(V)$, and $e \in \g$ is nilpotent if and only if $e$ is a nilpotent linear map on $V$.

\subsection{Actions of unipotent elements}\label{section:defunipotent} 

Let $q = 2^{\alpha}$ for some $\alpha \geq 0$. Then a cyclic group $C_q = \langle u \rangle$ has $q$ indecomposable $K[C_q]$-modules, up to isomorphism. We denote them by $V_1$, $\ldots$, $V_q$, where $\dim V_i = i$ and $u$ acts on $V_i$ as a single $i \times i$ unipotent Jordan block. For convenience we denote $V_0 = 0$.

Let $u$ be an element of a group $G$ and let $V$ be a finite-dimensional $K[G]$-module on which $u$ acts as a unipotent linear map. We denote $V^u := \{v \in V : uv = v\}$, which is the fixed point space of $u$ on $V$. Note that $\dim V^u$ is the number of Jordan blocks of $u$ on $V$. 

If $u \in \GL(V)$ is unipotent, we denote by $K[u]$ the group algebra of $\langle u \rangle$. Then $V \downarrow K[u] = V_{d_1}^{n_1} \oplus \cdots \oplus V_{d_t}^{n_t}$, where $0 < d_1 < \cdots < d_t$ are the Jordan block sizes of $u$ on $V$, and $n_i$ is the multiplicity of a Jordan block of size $d_i$.

\subsection{Actions of nilpotent elements}\label{section:defnilpotent} Let $q = 2^{\alpha}$ for some $\alpha \geq 0$. We denote by $\mathfrak{w}_q$ the abelian $2$-Lie algebra over $K$ generated by a nilpotent element $e$ with $e^{[2^{\alpha}]} = 0$ and $e^{[2^{\alpha-1}]} \neq 0$. Then as a $K$-vector space $$\mathfrak{w}_q = \bigoplus_{0 \leq i < \alpha} \langle e^{[2^i]} \rangle.$$ There are a total of $q$ indecomposable $\mathfrak{w}_q$-modules, up to isomorphism. We denote them by $W_1$, $\ldots$, $W_q$, where $\dim W_i = i$ and $e$ acts on $W_i$ as a single $i \times i$ nilpotent Jordan block. For convenience we denote $W_0 = 0$.

If $V$ is a finite-dimensional module for a Lie algebra $\g$ and $e \in \g$ acts on $V$ as a nilpotent element, we denote $V^e := \{v \in V : ev = 0\}$. Then $\dim V^e$ is the number of Jordan blocks of $e$ on $V$.

If $e \in \mathfrak{gl}(V)$ is nilpotent, we denote by $K[e]$ the $2$-Lie algebra generated by $e$. Then $K[e] \cong \mathfrak{w}_q$, where $q$ is the smallest power of two such that $e^q = 0$. Furthermore $V \downarrow K[e] = W_{d_1}^{n_1} \oplus \cdots \oplus W_{d_t}^{n_t}$, where $0 < d_1 < \cdots < d_t$ are the Jordan block sizes of $e$ on $V$, and $n_i$ is the multiplicity of a Jordan block of size $d_i$.

\section{Jordan normal forms on subspaces and quotients}\label{section:prelim}

In the proofs of our main results, the basic tool that we use to describe Jordan block sizes are the following two lemmas.

\begin{lemma}[{\cite[Lemma 3.3]{KorhonenJordanGood}}]\label{jordanrestrictionNIL}
Let $e \in \mathfrak{gl}(V)$ be a nilpotent linear map. Suppose that $W \subseteq V$ is a subspace invariant under $e$ such that $\dim V/W = 1$. Let $m \geq 1$ be such that $\operatorname{Ker} e^{m-1} \subseteq W$ and $\operatorname{Ker} e^{m} \not\subseteq W$. Then \begin{align*}
		V &\cong W_{m} \oplus V' \\
		W &\cong W_{m-1} \oplus V' \end{align*} for some $K[e]$-module $V'$.\end{lemma}

\begin{lemma}[{\cite[Lemma 3.4]{Korhonen2020Hesselink}}]\label{jordanquotientNIL}
Let $e \in \mathfrak{gl}(V)$ be a nilpotent linear map. Suppose that $W \subseteq V$ is a subspace invariant under $e$ such that $\dim W = 1$. Let $m \geq 1$ be such that $\operatorname{Im} e^{m-1} \supseteq W$ and $\operatorname{Im} e^{m} \not\supseteq W$. Then \begin{align*}
		V &\cong W_m \oplus V' \\
		V/W &\cong W_{m-1} \oplus V' \end{align*} for some $K[e]$-module $V'$.\end{lemma}
		
\begin{remark}
Recall that we define $W_0 = 0$. Thus if $m = 1$ in Lemma \ref{jordanrestrictionNIL}, we have $$V \cong W_1 \oplus W,$$ and so the Jordan normal form of $e$ on $W$ is given by removing a Jordan block of size $1$ from the Jordan normal form of $e$ on $V$. (Similar remarks apply to Lemma \ref{jordanquotientNIL}.)
\end{remark}

\section{Decomposition of tensor products and symmetric squares}\label{section:tensor}

Let $q = 2^{\alpha}$ for some $\alpha \geq 0$. In this section, we state some basic results about the decomposition of tensor products, exterior squares, and symmetric squares of $K[C_q]$-modules and $\mathfrak{w}_q$-modules.

\subsection{Unipotent case} Let $u$ be a generator for the cyclic group $C_q$. To describe the indecomposable summands of tensor products $V \otimes W$ of $K[C_q]$-modules, it is clear that it suffices to do so in the case where $V$ and $W$ are indecomposable. 

The decomposition of $V_m \otimes V_n$ into indecomposable summands has been studied extensively in all characteristics, see for example (in chronological order) \cite{Srinivasan}, \cite{Ralley}, \cite{McFall}, \cite{Renaud}, \cite{Norman}, \cite{Hou}, \cite{Norman2}, \cite{IimaIwamatsu2009}, and \cite{Barry}. Although there is no closed formula in general, there are various recursive formulae which can be used to determine the indecomposable summands efficiently. We are assuming $\chr K = 2$, in which case we have the following result.

\begin{lause}[{\cite[(2.5a)]{GreenModular}, \cite[Lemma 1, Corollary 3]{GowLaffey}}]\label{thm:tensordecompchar2}
Let $0 < m \leq n \leq q$ and suppose that $q/2 < n \leq q$. Then the following statements hold:
\begin{enumerate}[\normalfont (i)]
\item If $n = q$, then $V_m \otimes V_n \cong V_q^m$ as $K[C_q]$-modules.
\item If $m+n > q$, then $V_m \otimes V_n \cong V_q^{n+m-q} \oplus (V_{q-n} \otimes V_{q-m})$ as $K[C_q]$-modules.
\item If $m+n \leq q$, then $V_m \otimes V_n \cong V_{q-d_t} \oplus \cdots \oplus V_{q-d_1}$ as $K[C_q]$-modules, where $V_m \otimes V_{q-n} \cong V_{d_1} \oplus \cdots \oplus V_{d_t}$.
\end{enumerate}
\end{lause}

With Theorem \ref{thm:tensordecompchar2}, every tensor product $V_m \otimes V_n$ is either described explicitly (case (i)), or in terms of another tensor product $V_{m'} \otimes V_{n'}$ with $0 < m' \leq n' < n$.  Thus by repeated applications of Theorem \ref{thm:tensordecompchar2}, we can rapidly find the indecomposable summands of $V_m \otimes V_n$ for any given $m$ and $n$.

For exterior squares and symmetric squares, similarly it suffices to consider the indecomposable case. This follows from the fact that for all $K[C_q]$-modules $V$ and $W$, we have isomorphisms \begin{align*} \wedge^2(V \oplus W) &\cong \wedge^2(V) \oplus \wedge^2(W) \oplus V \otimes W \\ S^2(V \oplus W) &\cong S^2(V) \oplus S^2(W) \oplus V \otimes W\end{align*} as $K[C_q]$-modules. For the decomposition of $\wedge^2(V_n)$ and $S^2(V_n)$, we have the following results.

\begin{lause}[{\cite[Theorem 2]{GowLaffey}}]\label{thm:unipext1}
Suppose that $q/2 < n \leq q$. Then we have $$\wedge^2(V_n) \cong \wedge^2(V_{q-n}) \oplus V_q^{n-q/2-1} \oplus V_{3q/2-n}$$ as $K[C_q]$-modules.
\end{lause}

\begin{lause}[{\cite[Theorem 1.3]{KorhonenSymExt2021}}]\label{thm:unipsym1}
Suppose that $q/2 < n \leq q$. Then we have $$S^2(V_n) \cong \wedge^2(V_{q-n}) \oplus V_q^{n-q/2} \oplus V_{q/2}$$ as $K[C_q]$-modules.
\end{lause}

Similarly to Theorem \ref{thm:tensordecompchar2}, with Theorem \ref{thm:unipext1} and Theorem \ref{thm:unipsym1} we can quickly decompose $\wedge^2(V_n)$ and $S^2(V_n)$ for any given $n$.

\begin{lemma}\label{lemma:fixpdimsymwedge}
Let $n > 0$ be an integer. Then the following hold:
	\begin{enumerate}[\normalfont (i)]
		\item $\dim \wedge^2(V_n)^u = \lfloor \frac{n}{2} \rfloor$.
		\item $\dim S^2(V_n)^u = \lfloor \frac{n}{2} \rfloor + 1$.
	\end{enumerate}
\end{lemma}

\begin{proof}
We first consider (i). If $n = 1$, then $\wedge^2(V_1) = 0$ has dimension $0$ and thus (i) holds. Suppose then that $n > 1$ and proceed by induction on $n$. Let $q$ be a power of two such that $q/2  < n \leq q$. If $n = q$, it follows from Theorem \ref{thm:unipext1} that $\wedge^2(V_n) \cong V_q^{n-q/2-1} \oplus V_{q/2}$, so $\dim \wedge^2(V_n)^u = n - q/2 = n/2$. If $q/2 < n < q$, we have $$\wedge^2(V_n) \cong \wedge^2(V_{q-n}) \oplus V_{q}^{n-q/2-1} \oplus V_{3q/2-n}$$ by Theorem \ref{thm:unipext1}. Then by induction $$\dim \wedge^2(V_n)^u = \left\lfloor \frac{q-n}{2} \right\rfloor + n-q/2 = \left\lfloor \frac{n}{2} \right\rfloor,$$ as claimed by (i).

Next we will prove (ii). If $n = 1$, then $S^2(V_1) = V_1$ so (ii) holds. Suppose then that $n > 1$, and let $q$ be a power of two such that $q/2 < n \leq q$. If $n = q$, then $S^2(V_n) \cong V_q^{n-q/2} \oplus V_{q/2}$, and thus $\dim S^2(V_n)^u = n-q/2+1 = n/2+1$. Suppose then that $q/2 < n < q$. It follows from Theorem \ref{thm:unipsym1} and (i) that $$\dim S^2(V_n)^u = \left\lfloor \frac{q-n}{2} \right\rfloor + n-q/2+1 = \left\lfloor \frac{n}{2} \right\rfloor + 1,$$ as claimed by (ii).\end{proof}

\begin{lemma}\label{lemma:smallesblockunipS2}
Let $\ell > 0$ and define $\alpha = \nu_2(\ell)$. Then the smallest Jordan block size in $S^2(V_{2\ell})$ is $2^{\alpha}$, occurring with multiplicity one.
\end{lemma}

\begin{proof} (cf. \cite[Lemma 4.12]{Korhonen2020Hesselink}) If $\ell = 2^{\alpha}$, it follows from Theorem \ref{thm:unipsym1} that $S^2(V_{2\ell}) = V_{2^{\alpha+1}}^{2^{\alpha}} \oplus V_{2^{\alpha}}$, so the claim holds. If $\ell \neq 2^{\alpha}$, we have $q/2 < 2\ell < q$ for some $q = 2^{\beta}$. Then \begin{equation}\label{eq:s2decompinlemma}S^2(V_{2\ell}) = \wedge^2(V_{q-2\ell}) \oplus V_q^{2\ell-q/2} \oplus V_{q/2}\end{equation} by Theorem \ref{thm:unipsym1}. 

Now $\nu_2(q-2\ell) = 2^{\alpha+1}$ since $q > 2^{\alpha+1}$, so by \cite[Lemma 4.12]{Korhonen2020Hesselink} the smallest Jordan block size in $\wedge^2(V_{q-2\ell})$ is $2^{\alpha}$, occurring with multiplicity one. Furthermore $2^{\alpha} < q/2$ because $q/2 < 2\ell < q$, so the lemma follows from~\eqref{eq:s2decompinlemma}.\end{proof}

\subsection{Nilpotent case} As in the previous section, to determine the indecomposable summands of tensor products, exterior squares, and symmetric squares of $\mathfrak{w}_q$-modules, it suffices to do so in the indecomposable case.

For the indecomposable summands of $W_m \otimes W_n$, it turns out that we get the same decomposition as for $V_m \otimes V_n$ in the unipotent case.

\begin{prop}[{\cite[Section III]{Fossum}, \cite[Corollary 5 (a)]{NormanTwoRelated}}]\label{prop:uninilsim}
Let $0 < n,m \leq q$ and suppose that we have $V_m \otimes V_n \cong V_{r_1} \oplus \cdots \oplus V_{r_t}$ as $K[C_q]$-modules for some $r_1, \ldots, r_t > 0$. Then $W_m \otimes W_n \cong W_{r_1} \oplus \cdots \oplus W_{r_t}$ as $\mathfrak{w}_q$-modules.
\end{prop}

Thus we can apply Theorem \ref{thm:tensordecompchar2} to find the decomposition of $W_m \otimes W_n$ into indecomposable summands.

Following \cite[p. 231]{GlasbyPraegerXiapart}, we call the \emph{consecutive-ones binary expansion} of an integer $n > 0$ the alternating sum $n = \sum_{1 \leq i \leq r} (-1)^{i+1} 2^{\beta_i}$ such that $\beta_1 > \cdots > \beta_r \geq 0$ and $r$ is minimal. (Here $\beta_{r-1} > \beta_r + 1$ if $r > 1$.)

The decomposition of $V_n \otimes V_n$ into indecomposable summands can be given explicitly in terms of the consecutive-ones binary expansion of $n$ \cite[Theorem 15]{GlasbyPraegerXiapart}. Such descriptions can also be given for $\wedge^2(V_n)$ and $S^2(V_n)$, by using Theorem \ref{thm:unipext1} and Theorem \ref{thm:unipsym1}. 

For the decomposition of $W_n \otimes W_n$, $\wedge^2(W_n)$, and $S^2(W_n)$ we have the following result.

\begin{lause}[{\cite[Theorem 1.6, Theorem 1.7, Theorem 3.7]{KorhonenSymExt2021}}]\label{thm:extsymnilpotent}
Let $n > 0$ be an integer, with consecutive-ones binary expansion $n = \sum_{1 \leq i \leq r} (-1)^{i+1} 2^{\beta_i}$, where $\beta_1 > \cdots > \beta_r \geq 0$. For $1 \leq k \leq r$ , define $d_k := 2^{\beta_k} + \sum_{k < i \leq r} (-1)^{k+i} 2^{\beta_i+1}$. Then \begin{align*}W_n \otimes W_n &\cong \bigoplus_{1 \leq k \leq r} W_{2^{\beta_k}}^{d_k} \\ \wedge^2(W_n) &\cong \bigoplus_{\substack{1 \leq k \leq r \\ \beta_k > 0}} W_{2^{\beta_k}-1}^{d_k/2} \\ S^2(W_n) &\cong W_1^{\lceil n/2 \rceil} \oplus \bigoplus_{\substack{1 \leq k \leq r \\ \beta_k > 0}} W_{2^{\beta_k}}^{d_k/2}\end{align*} as $\mathfrak{w}_q$-modules.\end{lause}

\section{Unipotent elements in \texorpdfstring{$\Sp(V)$}{Sp(V)}}\label{section:hesselinkunipotent}

Consider $\Sp(V)$ with $\dim V = 2\ell$. In this section, we will recall the description of unipotent conjugacy classes in $\Sp(V)$, due to Hesselink \cite{Hesselink}.  For more details, see for example \cite{Hesselink}, \cite[Chapter 4, Chapter 6]{LiebeckSeitzClass}, or \cite[Section 6]{Korhonen2020Hesselink}.

For a group $G$, a bilinear $K[G]$-module $(W, \beta)$ is a finite-dimensional $K[G]$-module $W$ equipped with a $G$-invariant bilinear form $\beta$. Two bilinear $K[G]$-modules $(W,\beta)$ and $(W', \beta')$ are said to be isomorphic, if there exists an isomorphism $W \rightarrow W'$ of $K[G]$-modules which is also an isometry.

Let $u,u' \in \Sp(V)$ be unipotent. Let $q$ be a power of two such that $u^q = (u')^q = 1$, so that $V \downarrow K[u]$ and $V \downarrow K[u']$ are $K[C_q]$-modules. Then $u$ and $u'$ are conjugate in $\Sp(V)$ if and only if $V \downarrow K[u] \cong V \downarrow K[u']$ as bilinear $K[C_q]$-modules. 

For $u \in \Sp(V)$ unipotent, it is clear that we can write $V \downarrow K[u] = U_1 \perp \cdots \perp U_t$, where $U_i$ are \emph{orthogonally indecomposable} $K[u]$-modules. Here orthogonally indecomposable means that if $U_i = U' \perp U''$ as $K[u]$-modules, then $U' = 0$ or $U'' = 0$. There are two basic types of orthogonally indecomposable $K[u]$-modules, which we can define as follows. (Similar definitions are given in \cite[Section 6.1]{LiebeckSeitzClass}.)

\begin{maar}\label{def:V2lUNIP}
For $\ell \geq 1$, we define the module $V(2\ell)$ as follows. Let $n = 2\ell$, and suppose that $V$ has basis $v_1$, $\ldots$, $v_{n}$ with $b(v_i,v_j) = 1$ if $i+j = n+1$ and $0$ otherwise. Define $u: V \rightarrow V$ by \begin{align*} u v_1 &= v_1 \\ u v_i &= v_{i} + v_{i-1} + \cdots + v_1 \text{ for all } 1 < i \leq \ell + 1 \\ u v_i &= v_i + v_{i-1} \text{ for all } \ell+1 < i \leq n.\end{align*} Then $u \in \Sp(V)$, and we define $V(2\ell)$ as the bilinear $K[u]$-module $V \downarrow K[u]$.
\end{maar}

\begin{maar}\label{def:WlUNIP}
For $\ell \geq 1$, we define the module $W(\ell)$ as follows. Let $n = 2\ell$, and suppose that $V$ has basis $v_1$, $\ldots$, $v_{n}$ with $b(v_i,v_j) = 1$ if $i+j = n+1$ and $0$ otherwise. Define $u: V \rightarrow V$ by \begin{align*} u v_1 &= v_1 \\ u v_i &= v_i + v_{i-1} + \cdots + v_1 \text{ for all } 1 < i \leq \ell \\ u v_{\ell+1} &= v_{\ell+1} \\ u v_i &= v_i + v_{i-1} \text{ for all } \ell+1 < i \leq n.\end{align*} Then $u \in \Sp(V)$, and we define $W(\ell)$ as the bilinear $K[u]$-module $V \downarrow K[u]$.
\end{maar}

The fact that the modules are isomorphic to those described by Hesselink \cite[Proposition 3.5]{Hesselink} is seen as follows.

	\begin{itemize}
		\item For the module $V(2\ell)$ in Definition \ref{def:V2lUNIP} we have $V(2\ell) \downarrow K[u] = V_{2\ell}$, so this agrees with \cite[Proposition 3.5]{Hesselink}.
		\item In Definition \ref{def:WlUNIP}, we have totally singular decomposition $V = W \oplus Z$, where $W = \langle v_1, \ldots, v_{\ell} \rangle$ and $Z = \langle v_{\ell+1}, \ldots, v_n \rangle$ are $K[u]$-modules with $W \cong Z \cong V_{\ell}$. From this it follows that $V$ is isomorphic to the module $W(\ell)$ defined by Hesselink, see for example \cite[Lemma 6.12]{Korhonen2020Hesselink}.
	\end{itemize}

The classification of unipotent conjugacy classes in $\Sp(V)$ is based on the following result.

\begin{lause}[{\cite[Proposition 3.5]{Hesselink}}]\label{thm:hesselinkUNIPindecomp}
Let $u \in \Sp(V)$ be unipotent such that $V \downarrow K[u]$ is orthogonally indecomposable. Then $V \downarrow K[u]$ is isomorphic to $V(2\ell)$ or $W(\ell)$, where $\dim V = 2\ell$.
\end{lause}

By Theorem \ref{thm:hesselinkUNIPindecomp}, for every unipotent $u \in \Sp(V)$ we have an orthogonal decomposition $$V \downarrow K[u] = U_1 \perp \cdots \perp U_t,$$ where for all $1 \leq i \leq t$ we have $U_i \cong V(2\ell_i)$ or $U_i \cong W(\ell_i)$ for some integer $\ell_i \geq 1$.

In general there can be several different ways to decompose $V \downarrow K[u]$ into orthogonally indecomposable summands, and even the number of summands is not uniquely determined. This is due to the fact that for even $m$, we have an isomorphism $$W(m) \perp V(m) \cong V(m) \perp V(m) \perp V(m)$$ of bilinear $K[u]$-modules.  However, there are normal forms which are uniquely determined, such as the \emph{Hesselink normal form} \cite[3.7]{Hesselink} \cite[Theorem 6.4]{Korhonen2020Hesselink} or the \emph{distinguished normal form} defined by Liebeck and Seitz in \cite[p. 61]{LiebeckSeitzClass}.

\section{Nilpotent elements in \texorpdfstring{$\mathfrak{sp}(V)$}{sp(V)}}\label{section:hesselinknilpotent}

We consider $G = \Sp(V)$ with Lie algebra $\mathfrak{sp}(V)$, where $\dim V = 2\ell$. We recall the classification of nilpotent orbits in $\mathfrak{sp}(V)$ due to Hesselink \cite{Hesselink}. For more details, we refer to \cite{Hesselink} and \cite[Chapter 4, Chapter 5]{LiebeckSeitzClass}.

For a nilpotent element $e \in \mathfrak{sp}(V)$, define the \emph{index function} $\chi_V: \Z_{\geq 0} \rightarrow \Z_{\geq 0}$ corresponding to $e$ by $$\chi_V(m) := \operatorname{min} \{n \geq 0 : b(e^{n+1}v, e^{n}v) = 0 \text{ for all } v \in \Ker e^m \}.$$ Let $0 < d_1 < \cdots < d_t$ be the Jordan block sizes of $e$, and let $n_i$ be the multiplicity of Jordan block size $d_i$ for $e$. By a result of Hesselink \cite[Theorem 3.8]{Hesselink}, the nilpotent orbit of $e$ is determined by the integers $d_i$, $n_i$, and the function $\chi_V$. 

Hesselink also proved that it suffices to only consider the values of $\chi_V$ on the Jordan block sizes $d_1$, $\ldots$, $d_t$.

\begin{lause}[{\cite[3.9]{Hesselink}}]\label{thm:hesselinkmainthm}
Let $e \in \mathfrak{sp}(V)$ be nilpotent with index function $\chi = \chi_V$ on $V$. Let $0 < d_1 < \cdots < d_t$ be the Jordan block sizes of $e$, and let $n_i$ be the multiplicity of a Jordan block of size $d_i$ for $e$. Then the following statements hold:
	\begin{enumerate}[\normalfont (i)]
		\item The nilpotent orbit of $e$ is determined by the symbol $({d_1}_{\chi(d_1)}^{n_1}, \ldots, {d_t}_{\chi(d_t)}^{n_t})$.
		\item $\chi(d_1) \leq \cdots \leq \chi(d_t)$ and $d_1 - \chi(d_1) \leq \cdots \leq d_t - \chi(d_t)$.
		\item $0 \leq \chi(d_i) \leq d_i/2$ for all $1 \leq i \leq t$.
		\item $\chi(d_i) = d_i/2$ if $n_i$ is odd.
	\end{enumerate}
\end{lause}

\begin{remark}
Conversely, consider integers $d_i$, $n_i$, $\chi(d_i)$ with $0 < d_1 < \cdots < d_t$ and $\sum_{i = 1}^t n_i d_i = \dim V$, such that conditions (ii) -- (iv) of Theorem \ref{thm:hesselinkmainthm} hold. Then there exists a nilpotent element $e \in \mathfrak{sp}(V)$ with corresponding symbol $({d_1}_{\chi(d_1)}^{n_1}, \ldots, {d_t}_{\chi(d_t)}^{n_t})$ \cite[3.9]{Hesselink}.
\end{remark}

Similarly to the unipotent case, we can phrase the classification in terms of bilinear modules. For a Lie algebra $\mathfrak{w}$, a \emph{bilinear $\mathfrak{w}$-module} $(W, \beta)$ is a finite-dimensional $\mathfrak{w}$-module $W$ equipped with a $\mathfrak{w}$-invariant bilinear form $\beta$, so $\beta(Xv,w) + \beta(v,Xw) = 0$ for all $X \in \mathfrak{w}$ and $v,w \in W$. Two bilinear $\mathfrak{w}$-modules $(W,\beta)$ and $(W',\beta')$ are said to be isomorphic, if there exists an isomorphism $W \rightarrow W'$ of $\mathfrak{w}$-modules which is also an isometry.

Let $e, e' \in \mathfrak{sp}(V)$ be nilpotent. Choose a power of two $q$ such that $e^q = (e')^q = 0$, so that $V \downarrow K[e]$ and $V \downarrow K[e']$ are $\mathfrak{w}_q$-modules. Then $e$ and $e'$ are conjugate under the action of $\Sp(V)$ if and only if $V \downarrow K[e] \cong V \downarrow K[e']$ as bilinear $\mathfrak{w}_q$-modules.

For nilpotent $e \in \mathfrak{sp}(V)$, it is clear that we can write $V \downarrow K[e] = V_1 \perp \cdots \perp V_t$, where $V_i$ are \emph{orthogonally indecomposable} $K[e]$-modules. (Here orthogonally indecomposable is defined similarly to the group case.) By the next lemma, the index function $\chi_V$ is determined by its restriction to the orthogonally indecomposable summands of $V \downarrow K[e]$.

\begin{lemma}[{\cite[Lemma 5.2]{LiebeckSeitzClass}}]\label{lemma:maxval}
Let $e \in \mathfrak{sp}(V)$ be nilpotent and assume $V \downarrow K[e] = W_1 \perp W_2$ as $K[e]$-modules. Then $\chi_V(m) = \max \{ \chi_{W_1}(m), \chi_{W_2}(m) \}$ for all $m \geq 0$.
\end{lemma}

The orthogonally indecomposable modules were classified by Hesselink. In the case of $\mathfrak{sp}(V)$, there are three types of orthogonally indecomposable modules, defined as follows. (Similar definitions are given in \cite[Section 5.1]{LiebeckSeitzClass}.)

\begin{maar}\label{def:V2l}
For $\ell \geq 1$, we define the module $V(2\ell)$ as follows. Let $n = 2\ell$, and suppose that $V$ has basis $v_1$, $\ldots$, $v_{n}$ with $b(v_i,v_j) = 1$ if $i+j = n+1$ and $0$ otherwise. Define $e: V \rightarrow V$ by \begin{align*} e v_1 &= 0 \\ e v_i &= v_{i-1} \text{ for all } 1 < i \leq n.\end{align*} Then $e \in \mathfrak{sp}(V)$, and we define $V(2\ell)$ as the bilinear $K[e]$-module $V \downarrow K[e]$.
\end{maar}

\begin{maar}\label{def:Wl}
For $\ell \geq 1$, we define the module $W(\ell)$ as follows. Let $n = 2\ell$, and suppose that $V$ has basis $v_1$, $\ldots$, $v_{n}$ with $b(v_i,v_j) = 1$ if $i+j = n+1$ and $0$ otherwise. Define $e: V \rightarrow V$ by \begin{align*} e v_1 &= 0, &  e v_i &= v_{i-1} \text{ for all } 1 < i \leq \ell. \\
e v_{\ell+1} &= 0, &  e v_i &= v_{i-1} \text{ for all } \ell+1 < i \leq n.\end{align*} Then $e \in \mathfrak{sp}(V)$, and we define $W(\ell)$ as the bilinear $K[e]$-module $V \downarrow K[e]$.
\end{maar}

\begin{maar}\label{def:Wkl}
For $\ell \geq 1$ and $0 < k < \ell/2$ we define the module $W_k(\ell)$ as follows. Let $n = 2\ell$, and suppose that $V$ has basis $v_1$, $\ldots$, $v_{n}$ with $b(v_i,v_j) = 1$ if $i+j = n+1$ and $0$ otherwise. Define $e: V \rightarrow V$ by \begin{align*} e v_1 &= 0, & e v_{\ell+1} &= 0 \\ e v_{n-k+1} &= v_{n-k} + v_k & e v_i &= v_{i-1} \text{ for all } i \not\in \{ 1,\ell+1,n-k+1 \}  \end{align*} Then $e \in \mathfrak{sp}(V)$, and we define $W_k(\ell)$ as the bilinear $K[e]$-module $V \downarrow K[e]$.
\end{maar}

The fact that the modules in Definition \ref{def:V2l} -- \ref{def:Wkl} agree with those described by Hesselink in \cite[Proposition 3.5]{Hesselink} is seen as follows. 

	\begin{itemize}
		\item In Definition \ref{def:V2l}, this is clear from the fact that $V \downarrow K[e] = W_{2\ell}$, so $V \downarrow K[e] = V(2\ell)$ as defined by Hesselink. 
		\item In Definition \ref{def:Wl}, we have $V \downarrow K[e] = W_{\ell} \oplus W_{\ell}$. It is easy to see that $b(ev,v) = 0$ for all $v \in V$, so $\chi_V(m) = 0$ for all $m \geq 0$. It follows from \cite[Proposition 3.5, Theorem 3.8]{Hesselink} that $V \downarrow K[e]$ is isomorphic to the module $W(\ell)$ defined by Hesselink.
		\item In Definition \ref{def:Wkl} we have used the representative given in \cite{KorhonenStewartThomas}, which gives $W_k(\ell)$ by \cite[Lemma 3.4]{KorhonenStewartThomas}.
	\end{itemize}

\begin{lause}[{\cite[Proposition 3.5]{Hesselink}}]\label{thm:hesselinkNILindecomp}
Let $e \in \mathfrak{sp}(V)$ be nilpotent such that $V \downarrow K[e]$ is orthogonally indecomposable. Then $V \downarrow K[e]$ is isomorphic to $V(2\ell)$, $W(\ell)$, or $W_k(\ell)$ for some $0 < k < \ell/2$. These modules are characterized by the following properties:

\begin{center}
	\begin{tabular}{l|l|l}
	  \multicolumn{1}{c|}{$V \downarrow K[e]$} & \multicolumn{1}{c|}{Jordan normal form on $V$} & \multicolumn{1}{c}{$\chi_V$} \\ \hline
	  $V(2\ell)$  & $W_{2\ell}$ & $\chi_V(2\ell) = \ell$ \\
		$W(\ell)$   & $W_{\ell} \oplus W_{\ell}$ & $\chi_V(\ell) = 0$ \\
		$W_k(\ell)$ \normalfont{($0 < k < \ell/2$)} & $W_{\ell} \oplus W_{\ell}$ & $\chi_V(\ell) = k$
	\end{tabular}
\end{center}

\end{lause}

By Theorem \ref{thm:hesselinkNILindecomp}, for every nilpotent $e \in \mathfrak{sp}(V)$ we have an orthogonal decomposition $$V \downarrow K[e] = U_1 \perp \cdots \perp U_t,$$ where for all $1 \leq i \leq t$ we have $U_i \cong V(2\ell_i)$, $U_i \cong W(\ell_i)$, or $U_i \cong W_{k_i}(\ell_i)$ ($0 < k_i < \ell_i/2$) for some integer $\ell_i \geq 1$. 

As in the unipotent case, the orthogonally indecomposable summands and their number is not uniquely determined. For example, by Lemma \ref{lemma:maxval} and Theorem \ref{thm:hesselinkmainthm} (i), we have isomorphisms \begin{align*} W(d) \perp V(d') &\cong W_{d'/2}(d) \perp V(d') & \text{ for } d > d' > 0 \text{ even}; \\ W(d) \perp V(d) &\cong V(d) \perp V(d) \perp V(d) & \text{ for } d > 0 \text{ even};\end{align*} of bilinear $K[e]$-modules.

We end this section with two observations about orthogonally indecomposable modules of the form $\sum_{1 \leq i \leq t} W(m_i)$.

\begin{lemma}\label{lemma:nilsingularcharac}
Let $e \in \mathfrak{sp}(V)$ be nilpotent. Then the following statements are equivalent:
	\begin{enumerate}[\normalfont (i)]
		\item There is an orthogonal decomposition $V \downarrow K[e] = \sum_{1 \leq i \leq t} W(m_i)$ for some integers $m_1$, $\ldots$, $m_t$.
		\item There is a totally singular decomposition $V = W \oplus Z$, where $W$ and $Z$ are $K[e]$-submodules of $V$.
		\item $b(ev,v) = 0$ for all $v \in V$.
	\end{enumerate}
\end{lemma}

\begin{proof}(cf. \cite[Lemma 6.12]{Korhonen2020Hesselink}) We prove that (i) $\Rightarrow$ (ii) $\Rightarrow$ (iii) $\Rightarrow$ (i).

(i) $\Rightarrow$ (ii): It is clear from Definition \ref{def:Wl} that each $W(m_i)$ has a decomposition $W(m_i) = W_i \oplus Z_i$, where $W_i$ and $Z_i$ are totally singular $K[e]$-submodules with $W_i \cong W_{m_i} \cong Z_i$. Therefore (i) implies (ii).

(ii) $\Rightarrow$ (iii): We have $b(ew,w) = b(ez,z) = 0$ for all $w \in W$ and $z \in Z$ since $W$ and $Z$ are $e$-invariant and totally singular. Thus $b(e(w+z),w+z) = b(ew,z) + b(ez,w) = 0$ for all $w \in W$ and $z \in Z$, since $e \in \mathfrak{sp}(V)$.

(iii) $\Rightarrow$ (i): Suppose that $b(ev,v) = 0$ for all $v \in V$. Then for the index function of $e$ we have $\chi_V(m) = 0$ for all $m \geq 1$. Therefore every orthogonally indecomposable summand of $V \downarrow K[e]$ must be of the form $W(m)$ for some $m \geq 1$ (Theorem \ref{thm:hesselinkNILindecomp} and Lemma \ref{lemma:maxval}), so (i) holds.\end{proof}

\begin{lemma}\label{lemma:nilsquaredecomp}
Let $e \in \mathfrak{sp}(V)$ be nilpotent. Then $V \downarrow K[e^2] = \sum_{1 \leq i \leq t} W(m_i)$ for some integers $m_1$, $\ldots$, $m_t$. 
\end{lemma}

\begin{proof}
We have $b(e^2v,v) + b(ev,ev) = 0$ for all $v \in V$ since $e \in \mathfrak{sp}(V)$, so $b(e^2v,v) = 0$ for all $v \in V$. Now the claim follows from Lemma \ref{lemma:nilsingularcharac}.\end{proof}

\section{Chevalley construction}\label{section:chevalley}

In this section, we will recall basics of the Chevalley construction of Lie algebras and simple algebraic groups, and in particular how it applies for groups of type $C_{\ell}$ in characteristic two. For more details, see for example \cite{SteinbergNotesAMS} or \cite[Chapter VII]{Humphreys}. We will also make some preliminary observations about the actions of unipotent and nilpotent elements on the adjoint Lie algebra $\g_{ad}$ of type $C_{\ell}$. 

\subsection{Chevalley construction}  Let $\mathfrak{g}$ be a finite-dimensional simple Lie algebra over $\C$. Fix a Cartan subalgebra $\mathfrak{h}$ of $\mathfrak{g}$, and let $\Phi$ be the corresponding root system, so $$\mathfrak{g} = \mathfrak{h} \oplus \bigoplus_{\alpha \in \Phi} \mathfrak{g}_\alpha,$$ where $\mathfrak{g}_\alpha = \{ X \in \mathfrak{g} : [H, X] = \alpha(H)X \text{ for all } H \in \mathfrak{h}. \}.$ The Killing form $\kappa$ is non-degenerate on $\mathfrak{h}$, so for all $\alpha \in \Phi$ there exists $H_\alpha' \in \mathfrak{h}$ such that $\kappa(H, H_\alpha') = \alpha(H)$ for all $H \in \mathfrak{h}$. We define $H_\alpha := \frac{2}{\kappa(H_\alpha', H_\alpha')} H_\alpha'$ for all $\alpha \in \Phi$. 

It was shown by Chevalley \cite[Th\'eor\`eme 1]{ChevalleyTohoku} that one can choose $X_\alpha \in \mathfrak{g}_\alpha$ such that the following properties hold:		

\begin{enumerate}[(a)]
\item $[X_\alpha, X_{-\alpha}] = H_{\alpha}$ for all $\alpha \in \Phi$,
\item If $\alpha, \beta \in \Phi$ and $\alpha+\beta \in \Phi$, then $[X_\alpha, X_\beta] = \pm(r+1) X_{\alpha+\beta}$, where $r \geq 0$ is the largest integer such that $\beta-r\alpha$ is a root.
\item If $\alpha, \beta \in \Phi$ and $\alpha+\beta \not\in \Phi$, then $[X_\alpha, X_\beta] = 0$.
\end{enumerate}

For a choice of root vectors $X_\alpha$ satisfying (a) -- (c) above, we define $\mathfrak{g}_\Z$ to be the $\Z$-span of all $X_\alpha$ and $H_\alpha$ for $\alpha \in \Phi$. Let $\Delta$ be a base for $\Phi$ and let $\Phi^+$ be the corresponding system of positive roots. Then $\{X_\alpha : \alpha \in \Phi\} \cup \{H_\alpha : \alpha \in \Delta\}$ is a $\Z$-basis of $\mathfrak{g}_\Z$, called a \emph{Chevalley basis} for $\mathfrak{g}$.

Fix a Chevalley basis for $\mathfrak{g}$. Let $\mathscr{U}_{\Z}$ be the corresponding Kostant $\Z$-form, which is the subring of the universal enveloping algebra generated by $1$ and $\frac{X_{\alpha}^k}{k!}$ for all $\alpha \in \Phi$ and $k \geq 1$.

Let $V$ be a faithful finite-dimensional $\mathfrak{g}$-module over $\C$. We denote the set of weights of $\mathfrak{h}$ in $V$ by $\Lambda(V)$. Then $\Z \Phi \subseteq \Z \Lambda(V) \subseteq \Lambda$, where $\Lambda$ is the weight lattice.

A \emph{lattice} in $V$ is the $\Z$-span of a basis of $V$. We say that a lattice $L \subseteq V$ is \emph{admissible}, if $L$ is $\mathscr{U}_{\Z}$-invariant. 

Let $F$ be an algebraically closed field of characteristic $p > 0$, and $L$ an admissible lattice in $V$. Set $V_F := F \otimes_{\Z} L$. Then for all $X \in \mathscr{U}_{\Z}$, the \emph{reduction modulo $p$} of $X$ is the $F$-linear map defined by $1 \otimes X' : V_F \rightarrow V_F$, where $X': L \rightarrow L$ is the action of $X$ on $L$.

In particular, for all $\alpha \in \Phi$ and $k \geq 0$ we have a linear map $X_{\alpha,k} : V_F \rightarrow V_F$ which is the reduction modulo $p$ of $X_{\alpha}^k/k!$ on $V_F$. Since $X_{\alpha}$ acts nilpotently on $V$, we can define for all $t \in F$ the root element $x_{\alpha}(t)$ as the exponential $x_{\alpha}(t) := \sum_{k \geq 0} t^k X_{\alpha,k}$. We have $x_{\alpha}(t) \in \GL(V_F)$, and the \emph{Chevalley group} (over $F$) corresponding to $V$ and $L$ is defined as $$G(V,L) := \langle x_{\alpha}(t) : \alpha \in \Phi, t \in F \rangle.$$ 

Then $G(V,L)$ is a simple algebraic group over $F$ with root system $\Phi$ \cite[Theorem 6]{SteinbergNotesAMS}. Furthermore, we have a maximal torus $$T = \langle h_{\alpha}(t) : \alpha \in \Delta, t \in K^\times \rangle,$$ where $h_{\alpha}(t)$ is defined as in \cite[Lemma 19, p. 22]{SteinbergNotesAMS}. Then the weights of $T$ on $V$ can be identified with $\Lambda(V)$, and the character group $X(T)$ can be identified with $\Z \Lambda(V)$ \cite[p. 39]{SteinbergNotesAMS}.

We say that $G(V,L)$ is \emph{simply connected} if $\Z \Lambda(V) = \Lambda$, and \emph{adjoint} if $\Z \Lambda(V) = \Z \Phi$.

\begin{lemma}[{\cite[Corollary 1, p. 41]{SteinbergNotesAMS}}]\label{lemma:GSCtoGAD}
Suppose that $V'$ is another faithful finite-dimensional $\mathfrak{g}$-module with admissible lattice $L'$, and denote the corresponding root elements in $G(V',L')$ by $x_{\alpha}'(t)$. Then:
	\begin{enumerate}[\normalfont (i)]
		\item If $G(V',L')$ is simply connected, there exists a morphism $\varphi: G(V',L') \rightarrow G(V,L)$ of algebraic groups with $x_{\alpha}'(t) \mapsto x_{\alpha}(t)$ for all $\alpha \in \Phi$ and $t \in F$.
		\item If $G(V',L')$ is adjoint, there exists a morphism $\varphi: G(V,L) \rightarrow G(V',L')$ of algebraic groups with $x_{\alpha}(t) \mapsto x_{\alpha}'(t)$ for all $\alpha \in \Phi$ and $t \in F$.
	\end{enumerate}
\end{lemma}
                                                                            
The Lie algebra of $G(V,L)$ is identified as follows. The stabilizer of $L$ in $\g$ and $\mathfrak{h}$ is given by \cite[Corollary 2, p. 16]{SteinbergNotesAMS} \begin{align*} \mathfrak{h}^L &= \{ H \in \mathfrak{h} : \mu(H) \in \Z \text{ for all } \mu \in \Lambda(V) \}, \\ \g^L &= \mathfrak{h}^L \oplus \bigoplus_{\alpha \in \Phi} \Z X_{\alpha}.\end{align*} In particular $\g^L$ and $\mathfrak{h}^L$ only depend on the weights in $V$, not the choice of the admissible lattice $L$. Furthermore, under the adjoint action $\g^L$ is an admissible lattice in $\g$ \cite[Proposition 27.2]{Humphreys}.

Then $\g^L$ is a $\Z$-Lie algebra which acts on $L$. The \emph{Chevalley algebra} $$\g(V,L) := F \otimes_{\Z} \g^L$$ is a Lie algebra over $F$ which acts faithfully on $V_F$, and $\g(V,L) \subseteq \mathfrak{gl}(V_F)$ is precisely the Lie algebra of $G(V,L)$. The adjoint action of $G(V,L)$ can be realized by applying the Chevalley construction with $V' = \g$ and $L' = \g^L$, and then taking the morphism $G(V,L) \rightarrow G(V',L')$ of Lemma \ref{lemma:GSCtoGAD} (ii).

Note that then the Lie algebra $\g(V,L)$ is a $G_{sc}$-module for a simply connected Chevalley group $G_{sc}$ with root system $\Phi$ (Lemma \ref{lemma:GSCtoGAD} (i)). In all cases, the structure of $\g(V,L)$ as a Lie algebra and as a $G_{sc}$-module is described by Hogeweij in \cite{Hogeweij}. Here the composition factors of $\g(V, L)$ are completely determined by $V$, but the submodule structure depends on the choice of the admissible lattice $L$.

In the simply connected case we have $\g^L = \g_{\Z}$. In the adjoint case we denote $\g_{\Z}^{ad} := \g^L$, so $\g_{\Z}^{ad}$ is the $\Z$-span of $\g_{\Z}$ and $\mathfrak{h}^{ad} := \{ H \in \mathfrak{h} : \alpha(H) \in \Z \text{ for all } \alpha \in \Phi \}$.

\subsection{Chevalley construction for type $C_{\ell}$}\label{subsection:chevalleyC}

We will now setup the Chevalley construction for groups of type $C_{\ell}$ over $K$. (Recall that by $K$ we always denote an algebraically closed field of characteristic two.) 

Let $V_{\C}$ be a $\C$-vector space of dimension $n = 2\ell$, with basis $v_1$, $\ldots$, $v_n$. We define a non-degenerate alternating bilinear form $(-,-)$ on $V_{\C}$ by $(v_i,v_{n-i+1}) = 1 = -(v_{n-i+1},v_i)$ for all $1 \leq i \leq \ell$, and $(v_i,v_j) = 0$ if $i+j \neq n+1$.

Let $\mathfrak{sp}(V_\C) = \{X \in \mathfrak{gl}(V_\C) : (Xv, w) + (v, Xw) = 0 \text{ for all } v,w \in V_{\C} \}$, so $\mathfrak{sp}(V_\C)$ is a simple Lie algebra of type $C_{\ell}$. Let $\mathfrak{h}_{\C}$ be the Cartan subalgebra formed by the diagonal matrices in $\mathfrak{sp}(V_\C)$. Then $\mathfrak{h}_{\C} = \{ \operatorname{diag}(h_1, \ldots, h_{\ell}, -h_{\ell}, \ldots, -h_1) : h_i \in \C \}$.

For $1 \leq i \leq \ell$, define linear maps $\varepsilon_i : \mathfrak{h}_{\C} \rightarrow \C$ by $\varepsilon_i(h) = h_i$ where $h$ is a diagonal matrix with diagonal entries $(h_1, \ldots, h_{\ell}, -h_{\ell}, \ldots, -h_1)$. Then $$\Phi = \{ \pm(\varepsilon_i \pm \varepsilon_j) : 1 \leq i < j \leq \ell \} \cup \{ \pm 2\varepsilon_i : 1 \leq i \leq \ell \}$$ is the root system for $\mathfrak{sp}(V_\C)$, and $$\Phi^+ = \{ \varepsilon_i \pm \varepsilon_j : 1 \leq i < j \leq \ell \} \cup \{ 2 \varepsilon_i : 1 \leq i \leq \ell \}$$ is a system of positive roots. The set of simple roots corresponding to $\Phi^+$ is $\Delta = \{\alpha_1, \ldots, \alpha_{\ell}\}$, where $\alpha_i = \varepsilon_i - \varepsilon_{i+1}$ for $1 \leq i < {\ell}$ and $\alpha_{\ell} = 2 \varepsilon_{\ell}$.

For all $i,j$ let $E_{i,j}$ be the linear endomorphism on $V_{\C}$ such that $E_{i,j}(v_j) = v_i$ and $E_{i,j}(v_k) = 0$ for $k \neq j$. Throughout we will use the following Chevalley basis of $\mathfrak{sp}(V_\C)$, which is taken from \cite[Section 11, p. 38]{JantzenThesis}. \begin{align*} 
X_{\varepsilon_i - \varepsilon_j} &= E_{i,j} - E_{n-j+1,n-i+1} & \text{ for all } i \neq j, \\
X_{\varepsilon_i + \varepsilon_j} &= E_{j, n-i+1} + E_{i, n-j+1} & \text{ for all } i \neq j, \\
X_{-(\varepsilon_i + \varepsilon_j)} &= E_{n-j+1, i} + E_{n-i+1, j} & \text{ for all } i \neq j, \\
X_{2 \varepsilon_i} &= E_{i, n-i+1} & \text{ for all } i, \\
X_{-2 \varepsilon_i} &= E_{n-i+1, i} & \text{ for all } i. \\
H_{\alpha_i} &= [X_{\alpha_i}, X_{-\alpha_i}] & \text{ for all } 1 \leq i \leq \ell.
\end{align*}

Let $\Lambda \subset \mathfrak{h}_{\C}^*$ be the weight lattice. The maps $\varepsilon_1$, $\ldots$, $\varepsilon_{\ell}$ form $\Z$-basis for $\Lambda$, and for $1 \leq i \leq \ell$, the $i$th fundamental highest weight is equal to $\varpi_i := \varepsilon_1 + \cdots + \varepsilon_i$.

As in the previous subsection, we denote by $\g_{\Z}$ be the $\Z$-span of the Chevalley basis above. Then $\g_{\Z}^{ad}$ is the $\Z$-span of $\g_{\Z}$ and $H$, where $H \in \mathfrak{h}_{\C}$ is the diagonal matrix $H = \diag(1/2, \ldots, 1/2, -1/2, \ldots, -1/2)$. In terms of the Chevalley basis, we have $$H = \frac{1}{2} \left(H_{\alpha_1} + 2 H_{\alpha_2} + \cdots + \ell H_{\alpha_{\ell}} \right).$$ 

\subsection{Simply connected groups of type $C_{\ell}$}\label{subsection:chevalleySP} Let $V_{\Z}$ be the $\Z$-span of the basis $v_1$, $\ldots$, $v_n$ of $V_{\C}$. It is clear that $V_{\Z}$ is an admissible lattice. Denote $V := K \otimes_{\Z} V_{\Z}$. By abuse of notation we denote $v_i := 1 \otimes v_i$ for all $1 \leq i \leq n$, so $v_1$, $\ldots$, $v_n$ is a basis of $V$. The alternating bilinear form $(-,-)$ induces an alternating bilinear form $b: V \times V \rightarrow K$ on $V$, with $b(v_i,v_j) = 1$ if $i+j = n+1$ and $b(v_i,v_j) = 0$ otherwise.

Then the Chevalley group corresponding to $V_{\C}$ and $V_{\Z}$ is simply connected of type $C_{\ell}$, and it is precisely the symplectic group $\Sp(V)$ corresponding to $b$ \cite[5]{Ree}. We denote $G_{sc} = \Sp(V)$. Then the Lie algebra of $G_{sc}$ is $$\g_{sc} = \mathfrak{sp}(V) = \{X \in \mathfrak{gl}(V) : b(Xv,w) + b(v,Xw) = 0 \text{ for all } v,w \in V\}.$$

We denote by $\g_{ad} := K \otimes_{\Z} \g_{\Z}^{\ad}$ the adjoint Lie algebra of type $C_{\ell}$. For any $\mathfrak{sp}(V_{\C})$-module $W$ with admissible lattice $L$, we will consider $K \otimes_{\Z} L$ as a $G_{sc}$-module via the action provided by Lemma \ref{lemma:GSCtoGAD} (i).

\subsection{Lie algebra of adjoint type $C_{\ell}$}\label{subsection:symsquareC} To compute with the action of $G_{sc} = \Sp(V)$ on $\g_{ad}$, it will be convenient to use the following well-known identification of $\mathfrak{sp}(V_{\C})$ with the symmetric square $S^2(V_{\C})$.

\begin{lemma}\label{lemma:overCisomorphism}
For $x,y \in V_{\C}$, define a linear map $\psi_{x,y}: V_{\C} \rightarrow V_{\C}$ by $v \mapsto (y,v)x + (x,v)y$. Then we have an isomorphism $S^2(V_{\C}) \rightarrow \mathfrak{sp}(V_{\C})$ of $\mathfrak{sp}(V_{\C})$-modules, defined by $xy \mapsto \psi_{x,y}$ for all $x,y \in V_{\C}$.
\end{lemma}

\begin{proof}We have an isomorphism $\tau: V_{\C} \otimes V_{\C}^* \rightarrow \mathfrak{gl}(V_{\C})$ of $\mathfrak{gl}(V_{\C})$-modules, where $\tau(v \otimes f)$ is the linear map $w \mapsto f(w)v$ for all $v \in V_{\C}$ and $f \in V_{\C}^*$.  Moreover, we have an isomorphism $\tau': V_{\C} \otimes V_{\C} \rightarrow V_{\C} \otimes V_{\C}^*$ of $\mathfrak{sp}(V_{\C})$-modules defined by $x \otimes y \mapsto x \otimes f_y$, where $f_y(v) = (y,v)$ for all $v \in V$. 

Now identifying $S^2(V_{\C}) \subseteq V_{\C} \otimes V_{\C}$ via $xy \mapsto x \otimes y + y \otimes x$, the restriction of $\tau\tau'$ to $S^2(V_{\C})$ is a map $\psi: S^2(V_{\C}) \rightarrow \mathfrak{gl}(V_{\C})$ defined by $xy \mapsto \psi_{x,y}$. Then $\psi$ is an injective morphism of $\mathfrak{sp}(V_{\C})$-modules, and it is clear that $\psi_{x,y} \in \mathfrak{sp}(V_{\C})$ for all $x,y \in V_{\C}$. Since $S^2(V_{\C})$ and $\mathfrak{sp}(V_{\C})$ have the same dimension, we conclude that $\psi$ defines an isomorphism $S^2(V_{\C}) \rightarrow \mathfrak{sp}(V_{\C})$ of $\mathfrak{sp}(V_{\C})$-modules.\end{proof}

We denote by $\psi: S^2(V_{\C}) \rightarrow \mathfrak{sp}(V_{\C})$ the isomorphism $\psi(xy) = \psi_{x,y}$ as in Lemma \ref{lemma:overCisomorphism}. Let $1 \leq i, j \leq \ell$ with $i \neq j$. Then under the map $\psi$, the root vectors $X_{\pm \varepsilon_i \pm \varepsilon_j}$ in the Chevalley basis correspond to elements of $S^2(V_{\C})$ as follows. \begin{align*}
		\psi\left(v_i v_{n-j+1}\right) &= - X_{\varepsilon_i - \varepsilon_j} \\
		\psi\left(v_i v_j\right) &= X_{\varepsilon_i + \varepsilon_j} \\
		\psi\left(v_{n-i+1} v_{n-j+1}\right) &= -X_{-(\varepsilon_i + \varepsilon_j)} \\
		\psi\left(\frac{1}{2} v_i^2\right) &= X_{2\varepsilon_i} \\
		\psi\left(\frac{1}{2} v_{n-i+1}^2\right) &= -X_{-2\varepsilon_i}
	\end{align*} Furthermore, define $$\delta := \frac{1}{2} \sum_{i = 1}^{\ell} v_i v_{n-i+1}.$$ Then $\psi(\delta) = -H$.
	
We define \begin{align*} L_{sc} &:= \Z \text{-span of } v_iv_j \text{ and } \frac{1}{2}v_i^2 \text{ for all } 1 \leq i,j \leq n; \\ L_{ad} &:= \Z \text{-span of } L_{sc} \text{ and } \delta;\end{align*} so that $\psi(L_{sc}) = \g_{\Z}$ and $\psi(L_{ad}) = \g_{\Z}^{ad}$.

Then $L_{sc} \subset L_{ad}$ are admissible lattices in $S^2(V_{\C})$, since $\g_{\Z}$ and $\g_{\Z}^{ad}$ are admissible lattices in $\mathfrak{sp}(V_{\C})$. 

From the Chevalley construction $K \otimes_{\Z} \g_{\Z} = \g_{sc}$ and $K \otimes_{\Z} \g_{\Z}^{ad} = \g_{ad}$ as $G_{sc}$-modules, so $\psi$ induces isomorphisms \begin{equation}\label{eq:psiprimemorph}\begin{aligned} &\psi': \g_{sc} \rightarrow K \otimes_{\Z} L_{sc} \\ &\psi'': \g_{ad} \rightarrow K \otimes_{\Z} L_{ad}\end{aligned}\end{equation} of $G_{sc}$-modules. 

By \cite[Lemma 2.2]{Hogeweij} we have $\dim Z(\g_{sc}) = 1$ and $Z(\g_{sc})$ is spanned by $1 \otimes 2H$. Similarly by \cite[Table 1]{Hogeweij}, we have $\dim \g_{ad}/[\g_{ad},\g_{ad}] = 1$ and $[\g_{ad},\g_{ad}]$ is spanned by $1 \otimes v$ with $v \in \g_{\Z}$. Thus under the isomorphisms in~\eqref{eq:psiprimemorph}, we get \begin{equation}\label{eq:psiprimemorph2}\begin{aligned} \psi'(Z(\g_{sc})) &= \langle 1 \otimes 2\delta \rangle, \\ \psi''([\g_{ad},\g_{ad}]) &= \langle 1 \otimes v : v \in L_{sc} \rangle.\end{aligned}\end{equation}

Note that \begin{equation}\label{eq:ladlscmod2}L_{ad}/L_{sc} = \{L_{sc}, \delta + L_{sc}\} \cong \Z/2\Z,\end{equation} so for the linear map $\pi: K \otimes_{\Z} L_{sc} \rightarrow K \otimes_{\Z} L_{ad}$ induced by the inclusion $L_{sc} \subset L_{ad}$, we have $\Ker \pi = \langle 1 \otimes (2\delta) \rangle$ and $\operatorname{Im} \pi = \langle 1 \otimes v : v \in L_{sc} \rangle$.

\subsection{Action of unipotent elements on $\g_{ad}$, orthogonally indecomposable case}\label{subsection:unipotentroots} Let $G_{sc} = \Sp(V)$ as in the previous sections, and let $u \in \Sp(V)$ be a unipotent element. 

We will describe how to construct a representative for the conjugacy class of $u$ in terms of root elements, and how to compute the action of $u$ on $\g_{ad} \cong K \otimes_{\Z} L_{ad}$. We first do this in the case where $V \downarrow K[u]$ is orthogonally indecomposable, so $V \downarrow K[u] = V(2\ell)$ or $V \downarrow K[u] = W(\ell)$ (Theorem \ref{thm:hesselinkUNIPindecomp}). In this case it is well known that by replacing $u$ with a conjugate, we can take $$u = \begin{cases} x_{\alpha_1}(1) \cdots x_{\alpha_{\ell}}(1),& \text{ if } V \downarrow K[u] = V(2\ell). \\ x_{\alpha_1}(1) \cdots x_{\alpha_{\ell-1}}(1),& \text{ if } V \downarrow K[u] = W(\ell).\end{cases}$$ 

For $\alpha \in \Phi$, let $$x_{\Z}^{(\alpha)} := 1 + X_{\alpha} + \frac{X_{\alpha}^2}{2!} \in \mathscr{U}_{\Z}.$$ Then $x_{\alpha}(1)$ is the reduction modulo $p$ of the action of $x_{\Z}^{(\alpha)}$ on $V_{\Z}$. Furthermore, since $X_{\alpha}^3 \cdot S^2(V_{\C}) = 0$, the action of $x_{\alpha}(1)$ on $\g_{ad} \cong K \otimes_{\Z} L_{ad}$ is the reduction modulo $p$ of the action of $x_{\Z}^{(\alpha)}$ on $L_{ad}$. 

Therefore we define $$u_{\Z} := \begin{cases} x_{\Z}^{(\alpha_1)} \cdots x_{\Z}^{(\alpha_\ell)},& \text{ if } V \downarrow K[u] = V(2\ell). \\ x_{\Z}^{(\alpha_1)} \cdots x_{\Z}^{(\alpha_{\ell-1})},& \text{ if } V \downarrow K[u] = W(\ell).\end{cases}$$ If $V \downarrow K[u] = V(2\ell)$, we get \begin{align*} u_{\Z} \cdot v_1 &= v_1, \\ u_{\Z} \cdot v_i &= v_i + v_{i-1} + \cdots + v_1 \text{ if } 1 < i \leq \ell+1, \\ u_{\Z} \cdot v_i &= v_i - v_{i-1} \text{ if } \ell+1 < i \leq 2\ell.\end{align*} Similarly for $V \downarrow K[u] = W(\ell)$, we get \begin{align*} u_{\Z} \cdot v_1 &= v_1, \\ u_{\Z} \cdot v_i &= v_i + v_{i-1} + \cdots + v_1 \text{ if } 1 < i \leq \ell, \\ u_{\Z} \cdot v_{\ell+1} &= v_{\ell+1} \\ u_{\Z} \cdot v_i &= v_i - v_{i-1} \text{ if } \ell+1 < i \leq 2\ell.\end{align*} Thus the action of $u$ on $V$ is exactly as described in Definition \ref{def:V2lUNIP} or Definition \ref{def:WlUNIP}. 

Since $X_{\alpha_i}^3 \cdot S^2(V_{\C}) = 0$, for the action on $S^2(V_{\C})$ we have $$x_{\Z}^{(\alpha)} \cdot (vw) = (x_{\Z}^{(\alpha)}v)(x_{\Z}^{(\alpha)}w)$$ for all $v,w \in V_{\C}$. Thus $$u_{\Z} \cdot (vw) = (u_{\Z}v)(u_{\Z}w)$$ for all $v,w \in V_{\C}$.

\subsection{Action of unipotent elements on $\g_{ad}$, general case}\label{subsection:unipotentrootsGEN} For all unipotent elements $u \in G_{sc}$, in general we can proceed as follows. We have an orthogonal decomposition $V \downarrow K[u] = U_1 \perp \cdots \perp U_t$, where $\dim U_i > 0$ and $U_i$ is an orthogonally indecomposable $K[u]$-module for all $1 \leq i \leq t$. Write $\dim U_i = n_i = 2\ell_i$ for all $1 \leq i \leq t$.

Then $u = u_1 \cdots u_t$, where $u_i \in \Sp(U_i)$ is the action of $u$ on $U_i$. For all $1 \leq i \leq t$, we have $U_i \downarrow K[u_i] \cong W(\ell_i)$ or $U_i \downarrow K[u_i] \cong V(2\ell_i)$ (Theorem \ref{thm:hesselinkUNIPindecomp}). 

Relabel the basis $v_1$, $\ldots$, $v_n$ of $V_{\C}$ as $$v_1^{(1)}, \ldots, v_{n_1}^{(1)}, \ldots, v_1^{(t)}, \ldots, v_{n_t}^{(t)},$$ where $(v_i^{(j)}, v_{i'}^{(j')}) = 1$ if $j = j'$ and $i+i' = n_j$, and $0$ otherwise. Then the Cartan subalgebra $\mathfrak{h}$ consists of diagonal matrices of the form $$h = \operatorname{diag}(a_1^{(1)}, \ldots, a_{\ell_1}^{(1)}, -a_{\ell_1}^{(1)}, \ldots, -a_{1}^{(1)},\ \ldots,\ a_1^{(t)},\ldots, a_{\ell_t}^{(t)}, -a_{\ell_t}^{(t)}, \ldots, -a_{1}^{(t)}).$$ We define then for $1 \leq i \leq t$ and $1 \leq k \leq \ell_i$ the linear map $\varepsilon_k^{(i)} : \mathfrak{h} \rightarrow \C$ by $h \mapsto a_k^{(i)}$.

Denote by $V_{\Z}^{(i)}$ (resp. $V_{\C}^{(i)}$) the $\Z$-span (resp. $\C$-span) of $v_1^{(i)}, \ldots, v_{n_i}^{(i)}$, so we have \begin{align*}V_{\Z} &= V_{\Z}^{(1)} \oplus \cdots \oplus V_{\Z}^{(t)}, \\ V_{\C} &= V_{\C}^{(1)} \perp \cdots \perp V_{\C}^{(t)}.\end{align*} Since $V = K \otimes_{\Z} V_{\Z}$, we can assume that $U_i = K \otimes_{\Z} V_{\Z}^{(i)}$ for all $1 \leq i \leq t$. 

We have $\mathfrak{sp}(V_{\C}^{(i)}) \subseteq \mathfrak{sp}(V_{\C})$, and $\mathfrak{sp}(V_{\C}^{(i)})$ has a root system $\Phi^{(i)}$ of type $C_{\ell_i}$ with simple roots $\Delta^{(i)} = \{\beta_1, \ldots, \beta_{\ell_i} \}$, where $$\beta_j = \begin{cases} \varepsilon_j^{(i)} - \varepsilon_{j+1}^{(i)},& \text{ if } 1 \leq j < \ell_i, \\ 2 \varepsilon_{\ell_i}^{(i)},& \text{ if } j = \ell_i.\end{cases}$$ Note that $\{X_{\alpha} : \alpha \in \Phi^{(i)} \} \cup \{ H_{\alpha} : \alpha \in \Delta^{(i)} \}$ is a Chevalley basis of $\mathfrak{sp}(V_{\C}^{(i)})$, and the corresponding Kostant $\Z$-form $\mathscr{U}_{\Z}^{(i)}$ is a subring of $\mathscr{U}_{\Z}$. Then $V_{\Z}^{(i)}$ is an admissible lattice for $\mathfrak{sp}(V_{\C}^{(i)})$, and applying the Chevalley construction we get $\Sp(U_i)$.

As in the orthogonally indecomposable case, we define $$u_{\Z,i} := \begin{cases} x_{\Z}^{(\beta_1)} \cdots x_{\Z}^{(\beta_{\ell_i})},& \text{ if } U_i \downarrow K[u_i] = V(2\ell_i). \\ x_{\Z}^{(\beta_1)} \cdots x_{\Z}^{(\beta_{\ell_i-1})},& \text{ if } U_i \downarrow K[u_i] = W(\ell_i).\end{cases}$$ Then $u_{\Z,i} \in \mathscr{U}_{\Z}^{(i)}$, and the reduction modulo $p$ of the action of $u_{\Z,i}$ on $V_{\Z}^{(i)}$ is precisely $u_i \in \Sp(U_i)$.

Thus we can define $$u_{\Z} := u_{\Z,1} \cdots u_{\Z,t} \in \mathscr{U}_{\Z},$$ so that $u$ is the reduction modulo $p$ of the action of $u_{\Z}$ on $V_{\Z}$. Furthermore, the action of $u$ on $\g_{ad}$ is the reduction modulo $p$ of the action of $u_{\Z}$ on $L_{ad}$.

\subsection{Action of nilpotent elements on $\g_{ad}$}\label{subsection:nilpotentroots} Let $e \in \mathfrak{sp}(V)$ be nilpotent. Similarly to the unipotent case, we can find $e_{\Z} \in \g_{\Z}$ such that $e$ is the reduction modulo $p$ of the action of $e_{\Z}$ on $V_{\Z}$.

We first consider the orthogonally indecomposable case, so $V \downarrow K[e] = V(2\ell)$, $V \downarrow K[e] = W(\ell)$, or $V \downarrow K[e] = W_k(\ell)$ for some $0 < k < \ell/2$ (Theorem \ref{thm:hesselinkNILindecomp}). In this case we define $$e_{\Z} := \begin{cases} X_{\alpha_1} + \cdots + X_{\alpha_{\ell}}, & \text{ if } V \downarrow K[e] = V(2\ell). \\ X_{\alpha_1} + \cdots + X_{\alpha_{\ell-1}}, & \text{ if } V \downarrow K[e] = W(\ell). \\ X_{\alpha_1} + \cdots + X_{\alpha_{\ell-1}} + X_{2\varepsilon_k}, & \text{ if } V \downarrow K[e] = W_k(\ell).\end{cases}$$ Then for the reduction modulo $p$ of $e_{\Z}$, the action of $V = K \otimes_{\Z} V_{\Z}$ is exactly as in Definition \ref{def:V2l} -- \ref{def:Wkl}. (Here the expression for $W_k(\ell)$ is taken from \cite[Lemma 3.4]{KorhonenStewartThomas}.)

In the general case, we have an orthogonal decomposition $V \downarrow K[e] = U_1 \perp \cdots \perp U_t$, where $U_i$ is an orthogonally indecomposable $K[e]$-module, with $\dim U_i = 2\ell_i$. We have $e = e_1 + \cdots + e_t$, where $e_i \in \mathfrak{sp}(U_i)$ is the action of $e$ on $U_i$. In the notation of the previous section, we take $U_i = K \otimes_{\Z} V_{\Z}^{(i)}$. We can then define $$e_{\Z} := e_{\Z,1} + \cdots + e_{\Z,t},$$ where $e_{\Z,i} \in \g_{\Z} \cap \mathfrak{sp}(V_{\C}^{(i)})$ is defined as in the orthogonally indecomposable case, such that $e_i$ is the reduction modulo $p$ of $e_{\Z,i}$. 

Then $e$ is the reduction modulo $p$ of $e_{\Z}$. Moreover, given any irreducible $\mathfrak{sp}(V_{\C})$-module $W_{\C}$ and admissible lattice $L \subset W_{\C}$, the action of $e$ on $K \otimes_{\Z} L$ is precisely the reduction modulo $p$ of the action of $e_{\Z}$ on $L$.
  
\subsection{Elements of $\mathscr{U}_{\Z}$ acting nilpotently} In this section, we consider $X \in \mathscr{U}_{\Z}$ which act nilpotently on $S^2(V_{\C})$. Then by reduction modulo $p$, we get an action $\widehat{X}$ on $\g_{sc}$, and $\widetilde{X}$ on $\g_{ad}$. We will describe when $\operatorname{Im}(\widehat{X}) \supseteq Z(\g_{sc})$ and when $\operatorname{Ker}(\widetilde{X}) \subseteq [\g_{ad}, \g_{ad}]$. For our purposes this is mostly relevant for $X = (u_{\Z} - 1)^m$ in the unipotent case, and $X = e_{\Z}^m$ in the nilpotent case, for integer $m \geq 1$.

\begin{lemma}\label{lemma:nilpinUZ2}
Let $X \in \mathscr{U}_{\Z}$ be such that $X$ acts nilpotently on $S^2(V_{\C})$, and denote the action of $X$ on $\g_{sc}$ by $\widehat{X}$. Then $\operatorname{Im}(\widehat{X}) \supseteq Z(\g_{sc})$ if and only if there exists $v \in L_{sc}$ such that $X \cdot v \equiv 2 \delta \mod 2L_{sc}$.
\end{lemma}

\begin{proof}
Identifying $\g_{sc} = K \otimes_{\Z} L_{sc}$ via~\eqref{eq:psiprimemorph}, we have $Z(\g_{sc}) = \langle 1 \otimes 2 \delta \rangle$ by~\eqref{eq:psiprimemorph2}. Thus if $v \in L_{sc}$ is such that $X \cdot v \equiv 2 \delta \mod 2L_{sc}$, we have $\widehat{X} \cdot (1 \otimes v) = 1 \otimes 2 \delta$.

Conversely, suppose that $\operatorname{Im}(\widehat{X}) \supseteq Z(\g_{sc})$. Then there exists $w \in \g_{sc}$ with $\widehat{X} \cdot w = 1 \otimes 2\delta$. Since $\widehat{X} = 1 \otimes X'$ where $X'$ is the action of $X$ on $L_{sc}$, we can assume that $w = 1 \otimes v$ with $v \in L_{sc}$. Then $$1 \otimes 2\delta = \widehat{X} \cdot w = 1 \otimes (X \cdot v)$$ implies that $X \cdot v \equiv 2 \delta \mod 2L_{sc}$.\end{proof}

For actions of unipotent elements, Lemma \ref{lemma:nilpinUZ2} gives the following.

\begin{lemma}\label{lemma:imZforUZ}
Let $u \in G_{sc}$ be unipotent, and suppose that $u$ is the reduction modulo $p$ of $u_{\Z} \in \mathscr{U}_{\Z}$. Denote the action of $u$ on $\g_{sc}$ by $\widehat{u}$. Let $m \geq 1$ be an integer. Then $\operatorname{Im}(\widehat{u} - 1)^m \supseteq Z(\g_{sc})$ if and only if there exists $v \in L_{sc}$ such that $(u_{\Z} - 1)^m \cdot (v) \equiv 2 \delta \mod 2L_{sc}$.
\end{lemma}

\begin{proof}
Follows from Lemma \ref{lemma:nilpinUZ2}, with $X = (u_{\Z} - 1)^m$.
\end{proof}

\begin{lemma}\label{lemma:nilpinUZgen}
Let $X \in \mathscr{U}_{\Z}$ be such that $X$ acts nilpotently on $S^2(V_{\C})$, and denote the action of $X$ on $\g_{ad}$ by $\widetilde{X}$. Then the following statements are equivalent:
	\begin{enumerate}[\normalfont (i)]
		\item $\Ker(\widetilde{X}) \not\subseteq [\g_{ad}, \g_{ad}]$ .
		\item There exists $v \in L_{sc}$ such that $X \cdot (\delta + v) \in 2 L_{ad}$.
		\item There exists $v \in L_{sc}$ such that one of the following holds:
			\begin{enumerate}[\normalfont (a)] 
				\item $X \cdot (\delta + v) \equiv 0 \mod{2 L_{sc}}$.
				\item $X \cdot (\delta + v) \equiv 2\delta \mod{2 L_{sc}}$.
			\end{enumerate}
	\end{enumerate}
\end{lemma}

\begin{proof}
We first prove that (i) and (ii) are equivalent. We identify $\g_{ad} = K \otimes_{\Z} L_{ad}$ via~\eqref{eq:psiprimemorph}, in which case  $\g_{ad} = \langle 1 \otimes \delta, [\g_{ad}, \g_{ad}] \rangle$ as $K$-vector spaces. Here $[\g_{ad}, \g_{ad}]$ is the subspace spanned by $1 \otimes v$ with $v \in L_{sc}$, as seen in~\eqref{eq:psiprimemorph2}.

If there exists $v \in L_{sc}$ such that $X \cdot (\delta + v) \in 2 L_{ad}$, we have $$\widetilde{X} \cdot (1 \otimes \delta + 1 \otimes v) = 1 \otimes (X \cdot (\delta + v)) = 0,$$ so $\Ker(\widetilde{X}) \not\subseteq [\g_{ad}, \g_{ad}]$. 

Conversely, suppose that $\Ker(\widetilde{X}) \not\subseteq [\g_{ad}, \g_{ad}]$. Then there exists $w \in [\g_{ad}, \g_{ad}]$ such that $\widetilde{X} \cdot (1 \otimes \delta + w) = 0$. We have $\widetilde{X} = 1 \otimes X'$ where $X'$ is the action of $X$ on $L_{ad}$, so we can assume that $w = 1 \otimes v$ for some $v \in L_{sc}$. Then $$0 = \widetilde{X} \cdot (1 \otimes \delta + 1 \otimes v) = 1 \otimes (X \cdot (\delta+v))$$ implies that $X \cdot (\delta+v) \in 2L_{ad}$.

Next we show that (ii) and (iii) are equivalent. It follows from~\eqref{eq:ladlscmod2} that $2L_{ad}/2L_{sc} \cong L_{ad}/L_{sc} \cong \Z/2\Z$ contains only two elements, $2L_{sc}$ and $2\delta + 2L_{sc}$. Therefore $X \cdot (\delta + v) \in 2L_{ad}$ if and only if (iii)(a) or (iii)(b) holds.
\end{proof}

Next we make some observations that will allow us to reduce the proofs of our main results to the orthogonally indecomposable case. Continue with the notation as in Section \ref{subsection:unipotentrootsGEN}, so $V = U_1 \perp \cdots \perp U_t$ with $U_i = K \otimes_{\Z} V_{\Z}^{(i)}$, and $\mathscr{U}_{\Z}^{(i)}$ is the Kostant $\Z$-form for $\mathfrak{sp}(V_{\C}^{(i)})$. Define $$\delta_i = \frac{1}{2} \sum_{1 \leq j \leq \ell_i} v_j^{(i)}v_{2\ell_i-j+1}^{(i)}$$ for all $1 \leq i \leq t$. (Note that $\delta = \delta_1 + \cdots + \delta_t$.) Furthermore, for $1 \leq i \leq t$ we define \begin{align*} L_{sc}^{(i)} &:= \Z \text{-span of } v_j^{(i)}v_k^{(i)}\text{ and } \frac{1}{2}\left(v_j^{(i)}\right)^2 \text{ for all } 1 \leq j,k \leq n_i. \\ L_{ad}^{(i)} &:= \Z \text{-span of } L_{sc}^{(i)} \text{ and } \delta_i.\end{align*} Then \begin{equation}\label{eq:S2overZdecomp}L_{sc} = L_{sc}^{(1)} \oplus \cdots \oplus L_{sc}^{(t)} \oplus \bigoplus_{i \neq i'} V_{\Z}^{(i)} V_{\Z}^{(i')}\end{equation} as $\Z$-modules.

Moreover $L_{ad}^{(i)} \subset S^2(V_{\C}^{(i)})$ is an admissible lattice for the action of $\mathfrak{sp}(V_{\C}^{(i)})$, and $$\g_{ad}^{(i)} := K \otimes_{\Z} L_{ad}^{(i)}$$ is an adjoint Lie algebra of type $C_{\ell_i}$. Here $\Sp(U_i)$ acts on $\g_{ad}^{(i)}$ via the Chevalley construction.

\begin{lemma}\label{lemma:GADreduceNIL}
Let $X = X_1 + \cdots + X_t$, where $X_i \in \mathscr{U}_{\Z}^{(i)}$ acts nilpotently on $S^2(V_{\C})$ for all $1 \leq i \leq t$. Denote by $\widetilde{X}$ the linear map acting on $\g_{ad}$, given by reducing the action of $X$ on $L_{ad}$ modulo $p$. For $1 \leq i \leq t$, let $\widetilde{X_i}$ be the linear map acting on $\g_{ad}^{(i)}$, given by reducing the action of $X_i$ on $L_{ad}^{(i)}$ modulo $p$. 

Then the following statements hold.

	\begin{enumerate}[\normalfont (i)]
		\item $\Ker(\widetilde{X}) \not\subseteq [\g_{ad}, \g_{ad}]$ if and only if there exist $w_1$, $\ldots$, $w_t$ with $w_i \in L_{sc}^{(i)}$ such that one of the following holds:
			\begin{enumerate}[\normalfont (a)]
				\item $X_i \cdot (\delta_i + w_i) \equiv 0 \mod{2 L_{sc}^{(i)}}$ for all $1 \leq i \leq t$.
				\item $X_i \cdot (\delta_i + w_i) \equiv 2 \delta_i \mod{2 L_{sc}^{(i)}}$ for all $1 \leq i \leq t$.
			\end{enumerate}
		\item If $\Ker(\widetilde{X}) \not\subseteq [\g_{ad}, \g_{ad}]$, then $\Ker(\widetilde{X_i}) \not\subseteq [\g_{ad}^{(i)}, \g_{ad}^{(i)}]$ for all $1 \leq i \leq t$.
	\end{enumerate}
\end{lemma}

\begin{proof}
For (i), suppose first that there exist $w_1$, $\ldots$, $w_t$ with $w_i \in L_{sc}^{(i)}$ such that (i)(a) or (i)(b) holds. For $w = w_1 + \cdots + w_t$, we have $$X \cdot (\delta + w) = X_1 \cdot (\delta_1 + w_1) + \cdots + X_t \cdot (\delta_t + w_t).$$ Thus $X \cdot (\delta + w) \equiv 0 \mod{2 L_{sc}}$ if (a) holds, and $X \cdot (\delta + w) \equiv 2 \delta \mod{2 L_{sc}}$ if (b) holds. Therefore $\Ker(\widetilde{X}) \not\subseteq [\g_{ad}, \g_{ad}]$ by Lemma \ref{lemma:nilpinUZgen}.

Conversely, suppose that $\Ker(\widetilde{X}) \not\subseteq [\g_{ad}, \g_{ad}]$. Then by Lemma \ref{lemma:nilpinUZgen}, there exists $w \in L_{sc}$ such that $X \cdot (\delta + w) \equiv 0 \mod{2 L_{sc}}$ or $X \cdot (\delta + w) \equiv 2 \delta \mod{2 L_{sc}}$. By~\eqref{eq:S2overZdecomp} we can write $w = w_1 + \cdots + w_t + w'$, where $w_i \in L_{sc}^{(i)}$ for all $1 \leq i \leq t$ and $w' \in \bigoplus_{i \neq i'} V_{\Z}^{(i)} V_{\Z}^{(i')}$.

Note that $L_{sc}^{(1)} \oplus \cdots \oplus L_{sc}^{(t)}$ and $\bigoplus_{i \neq i'} V_{\Z}^{(i)} V_{\Z}^{(i')}$ are $\mathscr{U}_{\Z}^{(i)}$-invariant. Thus $$X \cdot (\delta + w) = X_1 \cdot (\delta_1 + w_1) + \cdots + X_t \cdot (\delta_t + w_t) + X \cdot w'$$ with $X_i \cdot (\delta_i + w_i) \in L_{ad}^{(i)}$ for all $1 \leq i \leq t$, and $X \cdot w' \in \bigoplus_{i \neq i'} V_{\Z}^{(i)} V_{\Z}^{(i')}$. Since $2\delta \in L_{sc}^{(1)} \oplus \cdots \oplus L_{sc}^{(t)}$, it follows that $X \cdot w' = 0$. Moreover we have assumed that $X \cdot (\delta + w) \equiv 0 \mod{2 L_{sc}}$ or $X \cdot (\delta + w) \equiv 2 \delta \mod{2 L_{sc}}$, so either (i)(a) or (i)(b) holds.

Claim (ii) follows from (i) and Lemma \ref{lemma:nilpinUZgen}.
\end{proof}

For the action of unipotent elements, we have the following result.

\begin{lemma}\label{lemma:GADreduce}
Let $u = u_1 \cdots u_t \in \Sp(V)$ be unipotent with $u_{\Z} = u_{\Z,1} \cdots u_{\Z,t} \in \mathscr{U}_{\Z}$ as in Section \ref{subsection:unipotentrootsGEN}. Denote the action of $u$ on $\g_{ad}$ by $\widetilde{u}$, and denote the action of $u_i$ on $\g_{ad}^{(i)}$ by $\widetilde{u_i}$.

Let $m \geq 1$ be an integer. Then the following statements hold.

\begin{enumerate}[\normalfont (i)]
	\item $\Ker(\widetilde{u}-1)^m \not\subseteq [\g_{ad},\g_{ad}]$ if and only if there exists $w_1$, $\ldots$, $w_t$ with $w_i \in L_{sc}^{(i)}$ such that one of the following holds:
		\begin{enumerate}[\normalfont (a)]
			\item $(u_{\Z,i} - 1)^m \cdot (\delta_i + w_i) \equiv 0 \mod{2 L_{sc}^{(i)}}$ for all $1 \leq i \leq t$.
			\item $(u_{\Z,i} - 1)^m \cdot (\delta_i + w_i) \equiv 2 \delta_i \mod{2 L_{sc}^{(i)}}$ for all $1 \leq i \leq t$.
		\end{enumerate}
	\item If $\Ker(\widetilde{u}-1)^m \not\subseteq [\g_{ad},\g_{ad}]$, then $\Ker(\widetilde{u_i} -1)^m \not\subseteq [\g_{ad}^{(i)}, \g_{ad}^{(i)}]$ for all $1 \leq i \leq t$.
\end{enumerate}
\end{lemma}

\begin{proof}It follows from Lemma \ref{lemma:nilpinUZgen} that $\Ker(\widetilde{u}-1)^m \not\subseteq [\g_{ad},\g_{ad}]$ if and only if there exists $w \in L_{sc}$ such that $(u_{\Z} - 1) \cdot (\delta + w) \equiv 0 \mod{2 L_{ad}}$. 

By~\eqref{eq:S2overZdecomp} we can write every $w \in L_{sc}$ in the form $w = w_1 + \cdots + w_t + w'$, where $w_i \in L_{sc}^{(i)}$ for all $1 \leq i \leq t$ and $w' \in \bigoplus_{i \neq i'} V_{\Z}^{(i)} V_{\Z}^{(i')}$. Since $u_{\Z,i}$ acts trivially on $V_{\Z}^{(i')}$ for $i \neq i'$, we have $$(u_{\Z} - 1)^m \cdot (\delta + w) = (u_{\Z,1} - 1) \cdot (\delta_1 + w_1) + \cdots + (u_{\Z,t} - 1) \cdot (\delta_t + w_t) + (u_{\Z} - 1) \cdot w'.$$ Here $(u_{\Z,i} - 1) \cdot (\delta_i + w_i) \in L_{ad}^{(i)}$ for all $1 \leq i \leq t$, and $(u_{\Z} - 1) \cdot w' \in \bigoplus_{i \neq i'} V_{\Z}^{(i)} V_{\Z}^{(i')}$. Thus the result follows similarly to the proof of Lemma \ref{lemma:GADreduceNIL}.\end{proof}

\section{Lie algebras of type \texorpdfstring{$C_{\ell}$}{C\_l} as \texorpdfstring{$G_{sc}$}{G\_sc}-modules}\label{section:typeClie}

In the notation of Section \ref{subsection:chevalleySP}, let $G_{sc} = \Sp(V)$ be a simply connected simple algebraic group of type $C_{\ell}$, with Lie algebra $\g_{sc} = \mathfrak{sp}(V)$. In this section, we will make some initial observations about the structure of $\g_{sc}$ and $\g_{ad}$ as $G_{sc}$-modules.

\begin{lemma}\label{lemma:LieSpS2}
Let $G_{sc} = \Sp(V)$ with Lie algebra $\g_{sc} = \mathfrak{sp}(V)$. Then $\g_{sc} \cong S^2(V)^*$ as $G_{sc}$-modules.
\end{lemma}

\begin{proof}A proof is given in \cite[5.4]{DowdSin}, alternatively this follows from the isomorphism $K \otimes_{\Z} L_{sc} \cong \g_{sc}$ given in Section \ref{subsection:symsquareC}.\end{proof}

\begin{lemma}\label{lemma:gmodZisgadgad}
We have $\g_{sc}/Z(\g_{sc}) \cong [\g_{ad}, \g_{ad}]$ as $G_{sc}$-modules.
\end{lemma}

\begin{proof}Let $G_{ad}$ be a group of adjoint type $C_{\ell}$, and let $\psi: G_{sc} \rightarrow G_{ad}$ be an isogeny as in Lemma \ref{lemma:GSCtoGAD}. Then the map $\D \psi: \g_{sc} \rightarrow \g_{ad}$ is a morphism of $G_{sc}$-modules, with $\Ker \D \psi = Z(\g_{sc})$ \cite[Lemma 2.2]{Hogeweij}. The image of $\D \psi$ contains all root elements of $\g_{ad}$, so by \cite[Table 1]{Hogeweij} it is equal to $[\g_{ad}, \g_{ad}]$. Therefore $\g_{sc}/ \Ker \D \psi = \g_{sc}/Z(\g_{sc}) \cong [\g_{ad}, \g_{ad}]$ as $G_{sc}$-modules.\end{proof}

Let $\varphi$ the linear map $\varphi: S^2(V) \rightarrow K$ defined by $\varphi(xy) = b(x,y)$ for all $x,y \in V$. Since $b$ is $G_{sc}$-invariant, it follows that $\varphi$ is a surjective morphism of $G_{sc}$-modules. 

\begin{lemma}\label{lemma:kerphiisdual}
Let $G_{sc} = \Sp(V)$ with Lie algebra $\g_{sc} = \mathfrak{sp}(V)$. Then $\Ker \varphi \cong \left( \g_{sc}/Z(\g_{sc}) \right)^*$ as $G_{sc}$-modules. 
\end{lemma}

\begin{proof}
It is clear that $\varphi$ is surjective, so we have a short exact sequence $$0 \rightarrow \Ker \varphi \rightarrow S^2(V) \rightarrow K \rightarrow 0$$ of $G_{sc}$-modules. This induces a short exact sequence $$0 \rightarrow K \rightarrow S^2(V)^* \rightarrow \left( \Ker \varphi \right)^* \rightarrow 0$$ of $G_{sc}$-modules. Now $S^2(V)^* \cong \g_{sc}$ as $G_{sc}$-modules (Lemma \ref{lemma:LieSpS2}), and $\g_{sc}$ has a unique trivial submodule since $Z(\g_{sc}) \cong K$ \cite[Table 1]{Hogeweij}. Thus we conclude that $\left( \Ker \varphi \right)^* \cong \g_{sc}/Z(\g_{sc})$ as $G_{sc}$-modules, from which the lemma follows.
\end{proof}

\begin{lemma}\label{lemma:uniqueuniserial}
Let $G = \Sp(V)$ be simply connected and simple of type $C_{\ell}$. Then:

	\begin{enumerate}[\normalfont (i)]
		\item There exists a uniserial $G$-module $W$ with $W = L_G(0)|L_G(2\varpi_1)|L_G(0)$.
		\item A $G$-module $W$ as in (i) is unique up to isomorphism.
		\item Let $u \in G$ be a unipotent element. Then $$\dim W^u = \begin{cases} \dim V^u + 1,& \text{ if } \dim V^u \text{ is odd.}\\ \dim V^u + 2,& \text{ if } \dim V^u \text{ is even.} \end{cases}$$
	\end{enumerate}
\end{lemma}

\begin{proof}
Let $G' = \SO(V')$ with $\dim V' = 2\ell+1$, so $G'$ is a simple algebraic group of adjoint type $B_{\ell}$. Let $\tau: G \rightarrow G'$ be an exceptional isogeny as in \cite[Theorem 28]{SteinbergNotesAMS}. We can embed $G'$ into a simple algebraic group $\SO(W)$ of type $D_{\ell+1}$ as the stabilizer of a nonsingular vector, see for example \cite[Section 6.8]{LiebeckSeitzClass}. Here $\dim W = 2\ell+2$, and as in \cite[Section 6.8]{LiebeckSeitzClass}, we identify $V' = \langle v \rangle^\perp \subset W$, where $v \in W$ is a nonsingular vector. 

It is straightforward to see that $W$ is a uniserial $G'$-module with $$W = L_{G'}(0)|L_{G'}(\varpi_1')|L_{G'}(0),$$ where $\varpi_1'$ is the first fundamental highest weight for $G'$. Then the twist of $W$ by $\tau$ is a uniserial $G$-module $W^\tau$ as in (i). For a unipotent element $u \in G$, the fixed point space of $u$ on the Frobenius twist $L_G(2\varpi_1) \cong L_G(\varpi_1)^{[1]}$ is the same as on $L_G(\varpi_1) \cong V$, because the Frobenius endomorphism preserves unipotent conjugacy classes. Therefore (iii) holds for $W$ by \cite[Lemma 3.8]{KorhonenUP}.

It remains to check that $W$ is unique. To this end, note first that $$\Ext_G^1(K, L_G(2 \varpi_1)) \cong \Hom_G(\wedge^2(V)^*, K)$$ \cite[II.2.14]{JantzenBook} and \cite[5.4]{DowdSin}. Here $\Hom_G(\wedge^2(V)^*, K) \cong \wedge^2(V)^G \cong K$, since $V$ has a unique $G$-invariant alternating bilinear form up to a scalar. Thus there exists a unique nonsplit extension $$0 \rightarrow K \rightarrow Z \rightarrow L_G(2\varpi_1) \rightarrow 0,$$ up to isomorphism of $G$-modules. 

Since $\Ext_G^1(K,K) = \Ext_G^2(K,K) = 0$ \cite[II.4.11]{JantzenBook} and $\Ext_G^1(K, L_G(2\varpi_1)) \cong K$, we have $\Ext_G^1(K,Z) \cong K$. Hence there exists a unique nonsplit extension $$0 \rightarrow Z \rightarrow W \rightarrow K \rightarrow 0,$$ up to isomorphism of $G$-modules. Every $W$ as in (i) is such an extension, so we conclude that $W$ is unique up to isomorphism.\end{proof}

\begin{lemma}\label{lemma:exactseqvarpi2gAD}
Let $G_{sc} = \Sp(V)$, so $G_{sc}$ is simply connected and simple of type $C_{\ell}$. Assume that $\ell$ is even. Then there is a short exact sequence $$0 \rightarrow L_G(\varpi_2) \rightarrow \g_{ad} \rightarrow W \rightarrow 0$$ of $G_{sc}$-modules, where $W$ is as in Lemma \ref{lemma:uniqueuniserial} (i).
\end{lemma}

\begin{proof}
As observed in \cite[5.6]{DowdSin}, in this case as a $G_{sc}$-module $\g_{ad}$ is uniserial with $\g_{ad} = L_G(\varpi_2)|L_G(0)|L_G(2\varpi_1)|L_G(0)$. Thus the result follows from Lemma \ref{lemma:uniqueuniserial} (ii).
\end{proof}

\begin{lemma}\label{lemma:regularl2mod4ineq}
Let $G = \Sp(V)$ be simply connected and simple of type $C_{\ell}$. Assume that $\ell \equiv 2 \mod{4}$ and let $u \in G$ be unipotent with $V \downarrow K[u] = V(2\ell)$. Then $\dim \g_{ad}^u \leq \dim \g_{sc}^u$.
\end{lemma}

\begin{proof}
It follows from Lemma \ref{lemma:fixpdimsymwedge} that $\dim \wedge^2(V)^u = \ell$. Because $\ell \equiv 2 \mod{4}$, the smallest Jordan block size of $u$ on $\wedge^2(V)$ is equal to $2$ \cite[Lemma 4.12]{Korhonen2020Hesselink}, and by \cite[Theorem B]{Korhonen2020Hesselink} we have $\dim L_G(\varpi_2)^u = \ell - 1$. 

For $W$ as in Lemma \ref{lemma:exactseqvarpi2gAD} (i), we have $\dim W^u = 2$ by Lemma \ref{lemma:exactseqvarpi2gAD} (iii). Thus it follows from Lemma \ref{lemma:exactseqvarpi2gAD} that $\dim \g_{ad}^u \leq \dim L_G(\varpi_2)^u + \dim W^u = \ell+1$. By Lemma \ref{lemma:LieSpS2} and Lemma \ref{lemma:fixpdimsymwedge} (ii) we have $\dim \g_{sc}^u = \ell+1$, so the result follows.\end{proof}

\begin{remark}
As a corollary of our results, we will see later that equality $\dim \g_{ad}^u = \dim \g_{sc}^u$ holds in Lemma \ref{lemma:regularl2mod4ineq}, without assumptions on $\ell$ (Corollary \ref{cor:regularcentralizerdim}).
\end{remark}

\section{Jordan block sizes of unipotent elements on \texorpdfstring{$\g_{sc}/Z(\g_{sc})$}{g\_sc/Z(g\_sc)}}\label{section:unipgmodZ}

Let $G_{sc} = \Sp(V)$ be simply connected of type $C_{\ell}$, and let $b$ be the $G_{sc}$-invariant alternating bilinear form defining $G_{sc}$. We have $\g_{sc} \cong S^2(V)^*$ by Lemma \ref{lemma:LieSpS2}, so for every unipotent element $u \in G$ the Jordan block sizes of $u$ on $\g_{sc}$ are known by the results described in Section \ref{section:tensor}. In this section, we will describe the Jordan block sizes of unipotent elements $u \in G$ on $\g_{sc}/Z(\g_{sc})$, in terms of Jordan block sizes of $u$ on $\g_{sc}$. 

Throughout this section, we will denote by $\varphi$ the linear map $\varphi: S^2(V) \rightarrow K$ defined by $\varphi(xy) = b(x,y)$ for all $x,y \in V$. Since $b$ is $G_{sc}$-invariant, it follows that $\varphi$ is a surjective morphism of $G_{sc}$-modules. By Lemma \ref{lemma:kerphiisdual}, the Jordan normal form of a unipotent element $u \in G$ on $\g_{sc}/Z(\g_{sc})$ is the same as on $\Ker \varphi$.

We will then describe the Jordan block sizes of unipotent elements $u \in G$ on $\Ker \varphi$. By Lemma \ref{jordanrestrictionNIL}, this amounts to finding the largest integer $m \geq 0$ such that $\Ker(\widetilde{u}-1)^m \subseteq \Ker \varphi$, where $\widetilde{u}$ is the action of $u$ on $S^2(V)$.

\begin{lemma}\label{lemma:unipS2decomp}
Let $u \in \Sp(V)$ be unipotent. Suppose that we have a orthogonal decomposition $V = U_1 \perp \cdots \perp U_t$ as $K[u]$-modules. Denote the action of $u$ on $S^2(V)$ by $\widetilde{u}$, and the action of $u$ on $S^2(U_i)$ by $\widetilde{u_i}$. Let $m \geq 0$ be an integer. 

Then $\Ker(\widetilde{u}-1)^m \subseteq \Ker \varphi$ if and only if $\Ker(\widetilde{u_i}-1)^m \subseteq \Ker \varphi$ for all $1 \leq i \leq t$.
\end{lemma}

\begin{proof}The orthogonal decomposition $V = U_1 \perp \cdots \perp U_t$ into $K[u]$-modules induces a decomposition $$S^2(V) = \bigoplus_{1 \leq i \leq t} S^2(U_i) \perp \bigoplus_{1 \leq i < j \leq t} U_iU_j$$ into $K[u]$-modules, where $U_iU_j = \{xy : x \in U_i, y \in U_j\}$. We have $U_iU_j \subseteq \Ker \varphi$ for all $1 \leq i < j \leq t$, from which the lemma follows.\end{proof}

With Lemma \ref{lemma:unipS2decomp}, we reduce to the case where $V \downarrow K[u]$ is orthogonally indecomposable. We first consider the case where $V \downarrow K[u] = V(2\ell)$, so $n = 2\ell$. Let $v_1$, $\ldots$, $v_n$ be a basis as in the definition of $V(2\ell)$ (Definition \ref{def:V2lUNIP}). Then $b(v_i,v_j) = 1$ if $i+j = n+1$, and $0$ otherwise. Furthermore, the action of $u$ is defined by \begin{align*}
uv_1 &= v_1, \\
uv_i &= v_i + v_{i-1} + \cdots + v_1 \text{ for all } 2 \leq i \leq \ell+1, \\
uv_i &= v_i + v_{i-1} \text{ for all } \ell+1 < i \leq n.
\end{align*} We will denote $v_j = 0$ for all $j \leq 0$ and $j > n$.

Define $\gamma = \sum_{1 \leq i \leq \ell} v_iv_{n+1-i}$. We have a short exact sequence $$0 \rightarrow V^{[2]} \rightarrow S^2(V) \rightarrow \wedge^2(V) \rightarrow 0,$$ where $V^{[2]}$ is the subspace generated by $v^2$ for all $v \in V$. As noted in \cite[Section 9]{Korhonen2020Hesselink}, the image of $\gamma$ in $\wedge^2(V)$ is fixed by the action of $\Sp(V)$. Thus $u \cdot \gamma = \gamma + x$ for some $x \in V^{[2]}$, and more precisely we have the following.

\begin{lemma}\label{lemma:unipS2delta}
Let $u$ and $\gamma$ be as above. Then $u \cdot \gamma = \gamma + \sum_{1 \leq i \leq \ell} v_i^2$, and $\gamma + v_{\ell+1}^2$ is fixed by the action of $u$.
\end{lemma}

\begin{proof}
Denote $s_i := \sum_{j \geq 0} v_{i-j} v_{n-i}$. First we note that \begin{align} \notag u \cdot \sum_{1 \leq i < \ell} v_i v_{n-i+1} &= \sum_{1 \leq i < \ell} \left( \sum_{j \geq 0} v_{i-j} \right)\left(v_{n-i+1} + v_{n-i} \right) \\ \notag &= \sum_{1 \leq i < \ell} \left(s_i + s_{i-1} + v_i v_{n-i+1}\right) \\ \label{eq:firstsum} &= s_{\ell-1} + \sum_{1 \leq i < \ell} v_i v_{n-i+1}. \end{align} By another calculation, we get \begin{equation} \label{eq:secondsum}  u \cdot v_{\ell}v_{\ell+1} = \left(\sum_{1 \leq j \leq \ell} v_j \right) \left(\sum_{1 \leq j \leq \ell+1} v_j \right) = s_{\ell-1} + v_{\ell}v_{\ell+1} +  \sum_{1 \leq j \leq \ell} v_j^2.\end{equation} Adding~\eqref{eq:firstsum} and~\eqref{eq:secondsum} together, we conclude that $u \cdot \gamma = \gamma + \sum_{1 \leq i \leq \ell} v_i^2$.

Since $u \cdot v_{\ell+1}^2 = \left(\sum_{1 \leq i \leq \ell+1} v_i\right)^2 = \sum_{1 \leq i \leq \ell+1} v_i^2$, it follows that $\gamma + v_{\ell+1}^2$ is fixed by the action of $u$.
\end{proof}

\begin{lemma}\label{lemma:S2indecompV2l}
Let $u \in \Sp(V)$ be unipotent such that $V \downarrow K[u] = V(2\ell)$. Let $\widetilde{u}$ be the action of $u$ on $S^2(V)$, and denote $\alpha = \nu_2(\ell)$. Then the following hold:
	\begin{enumerate}[\normalfont (i)]
		\item $\Ker(\widetilde{u}-1)^{2^{\alpha}-1} \subseteq \Ker \varphi$.
		\item $\Ker(\widetilde{u}-1)^{2^{\alpha}} \not\subseteq \Ker \varphi$.
	\end{enumerate}
\end{lemma}

\begin{proof}
By Lemma \ref{lemma:smallesblockunipS2} the smallest Jordan block size of $\widetilde{u}$ is $2^{\alpha}$, so (i) follows from Lemma \ref{jordanrestrictionNIL}. 

For (ii), we first consider the case where $\alpha = 0$, so $\ell$ is odd. In the notation used before the lemma, by Lemma \ref{lemma:unipS2delta} the vector $\gamma + v_{\ell+1}^2$ is fixed by the action of $u$. Furthermore $\varphi(\gamma + v_{\ell+1}^2) = \ell \neq 0$, so $\Ker(\widetilde{u}-1) \not\subseteq \Ker \varphi$, as claimed.

Suppose then that $\alpha > 0$. We have $V \downarrow K[u^{2^{\alpha}}] = V(\ell/2^{\alpha-1})^{2^{\alpha}}$ by \cite[Lemma 6.13]{Korhonen2020Hesselink}. Combining this with Lemma \ref{lemma:unipS2decomp} and the fact that $(\widetilde{u}-1)^{2^{\alpha}} = \widetilde{u}^{2^{\alpha}} - 1$, the claim follows from the $\alpha = 0$ case.\end{proof} 

\begin{lemma}\label{lemma:S2indecompWl}
Let $u \in \Sp(V)$ be unipotent such that $V \downarrow K[u] = W(\ell)$. Let $\widetilde{u}$ be the action of $u$ on $S^2(V)$, and denote $\alpha = \nu_2(\ell)$. Then the following hold:
	\begin{enumerate}[\normalfont (i)]
		\item $\Ker(\widetilde{u}-1)^{2^{\alpha}-1} \subseteq \Ker \varphi$.
		\item $\Ker(\widetilde{u}-1)^{2^{\alpha}} \not\subseteq \Ker \varphi$.
	\end{enumerate}
\end{lemma}

\begin{proof}
We have $V = W \oplus W^*$ as $K[u]$-modules, where $W$ and $W^*$ are totally isotropic subspaces. This induces a decomposition $$S^2(V) = S^2(W) \oplus S^2(W^*) \oplus WW^*$$ of $K[u]$-modules, where $WW^* \cong W \otimes W^*$ is the subspace $\{xy : x \in W, y \in W^*\}$. We have $S^2(W), S^2(W^*) \subseteq \Ker \varphi$ and the smallest Jordan block size in $W \otimes W^*$ is $2^{\alpha}$ \cite[Lemma 4.2]{KorhonenJordanGood}, so (i) follows from Lemma \ref{jordanrestrictionNIL}.

For (ii), we first consider the case where $\alpha = 0$, so $\ell$ is odd. Choose a basis $v_1$, $\ldots$, $v_{\ell}$ of $W$, and let $w_1$, $\ldots$, $w_{\ell}$ be the corresponding dual basis in $W^*$, so $b(v_i,w_j) = \delta_{i,j}$ for all $1 \leq i,j \leq \ell$. Then $\gamma' = \sum_{1 \leq i \leq \ell} v_iw_i$ is fixed by the action of $u$, see for example \cite[Lemma 3.7]{KorhonenJordanGood}. We have $\varphi(\gamma') = \ell \neq 0$, so $\Ker(\widetilde{u}-1) \not\subseteq \Ker \varphi$, as claimed.

Suppose then that $\alpha > 0$. We have $(\widetilde{u}-1)^{2^{\alpha}} = \widetilde{u}^{2^{\alpha}} - 1$, and furthermore $V \downarrow K[u^{2^{\alpha}}] = W(\ell/2^{\alpha})^{2^{\alpha}}$ by \cite[Lemma 6.12]{Korhonen2020Hesselink}. Thus as in the proof of Lemma \ref{lemma:S2indecompV2l}, the claim follows from Lemma \ref{lemma:unipS2decomp} and the $\alpha = 0$ case.\end{proof}

\begin{prop}\label{prop:unipKeraction}
Let $u \in \Sp(V)$ be unipotent, with orthogonal decomposition $$V \downarrow K[u] = \sum_{1 \leq i \leq t} W(m_i) \perp \sum_{1 \leq j \leq s} V(2k_j).$$ Denote by $\alpha \geq 0$ the largest integer such that $2^{\alpha} \mid m_i,k_j$ for all $i$ and $j$. Then \begin{align*}
		\g_{sc} &\cong V_{2^{\alpha}} \oplus V' \\
		\g_{sc}/Z(\g_{sc}) &\cong V_{2^{\alpha}-1} \oplus V' \end{align*} for some $K[u]$-module $V'$.
\end{prop}

\begin{proof}
Let $\widetilde{u}$ be the action of $u$ on $S^2(V)$. By Lemma \ref{lemma:LieSpS2}, Lemma \ref{lemma:kerphiisdual}, and Lemma \ref{jordanrestrictionNIL}, the proposition is equivalent to the statement that $\Ker(\widetilde{u}-1)^{2^{\alpha}-1} \subseteq \Ker \varphi$ and $\Ker(\widetilde{u}-1)^{2^{\alpha}} \not\subseteq \Ker \varphi$. 

By Lemma \ref{lemma:unipS2decomp}, Lemma \ref{lemma:S2indecompV2l}, and Lemma \ref{lemma:S2indecompWl}, we have $\Ker(\widetilde{u}-1)^{2^{\alpha}-1} \subseteq \Ker \varphi$. Next note that either $\alpha = \nu_2(m_i)$ for some $i$, or $\alpha = \nu_2(k_j)$ for some $j$.  It follows then from Lemma \ref{lemma:unipS2decomp}, together with Lemma \ref{lemma:S2indecompWl} (if $\alpha = \nu_2(m_i)$) or Lemma \ref{lemma:S2indecompV2l} (if $\alpha = \nu_2(k_j)$) that $\Ker(\widetilde{u}-1)^{2^{\alpha}} \not\subseteq \Ker \varphi$.\end{proof}

\section{Jordan block sizes of nilpotent elements on \texorpdfstring{$\g_{sc}/Z(\g_{sc})$}{g\_sc/Z(g\_sc)}}\label{section:nilgmodZ}

We continue with the setup of the previous section. Let $G_{sc} = \Sp(V)$ be simply connected of type $C_{\ell}$ with Lie algebra $\g_{sc} = \mathfrak{sp}(V)$. In this section, we will describe the Jordan block sizes of nilpotent elements $e \in \mathfrak{sp}(V)$ on $\g_{sc}/Z(\g_{sc}) \cong (\Ker \varphi)^*$ (Lemma \ref{lemma:kerphiisdual}). As in the previous section, we describe the Jordan block sizes of $e$ on $\Ker \varphi$ in terms of Jordan block sizes of $e$ on $S^2(V)$. Note that the Jordan normal form of $e$ on $S^2(V)$ is known by the results described in Section \ref{section:tensor}. 

We begin by reducing the calculation to the orthogonally indecomposable case. After this we consider the different orthogonally indecomposable $K[e]$-modules in turn, and by combining the results we obtain Proposition \ref{prop:nilKeraction}, which is analogous to Proposition \ref{prop:unipKeraction}.

\begin{lemma}\label{lemma:NILS2decomp}
Let $e \in \mathfrak{sp}(V)$ be unipotent. Suppose that we have a orthogonal decomposition $V = U_1 \perp \cdots \perp U_t$ as $K[e]$-modules. Denote the action of $e$ on $S^2(V)$ by $\widetilde{e}$, and the action of $e$ on $S^2(U_i)$ by $\widetilde{e_i}$. Let $m \geq 0$ be an integer. 

Then $\Ker(\widetilde{e})^m \subseteq \Ker \varphi$ if and only if $\Ker(\widetilde{e_i})^m \subseteq \Ker \varphi$ for all $1 \leq i \leq t$.
\end{lemma}

\begin{proof}Follows with the same proof as Lemma \ref{lemma:unipS2decomp}.\end{proof}

\begin{lemma}\label{lemma:S2indecompWlNIL}
Let $e \in \mathfrak{sp}(V)$ be nilpotent such that $V \downarrow K[e] = W(\ell)$. Let $\widetilde{e}$ be the action of $e$ on $S^2(V)$, and denote $\alpha = \nu_2(\ell)$. Then the following hold:
	\begin{enumerate}[\normalfont (i)]
		\item $\Ker(\widetilde{e})^{2^{\alpha}-1} \subseteq \Ker \varphi$.
		\item $\Ker(\widetilde{e})^{2^{\alpha}} \not\subseteq \Ker \varphi$.
	\end{enumerate}
\end{lemma}

\begin{proof}We proceed similarly to the proof of Lemma \ref{lemma:S2indecompWl}. By definition of $W(\ell)$, we have a totally singular decomposition $V = W \oplus W^*$, where $W$ and $W^*$ are $K[e]$-modules on which $e$ acts with a single Jordan block of size $\ell$. This induces a decomposition $$S^2(V) = S^2(W) \oplus S^2(W^*) \oplus WW^*$$ of $K[e]$-modules, where $WW^* \cong W \otimes W^*$. Thus $$\Ker(\widetilde{e}) = \Ker(\widetilde{e}_{S^2(W)}) \oplus \Ker(\widetilde{e}_{S^2(W)}) \oplus \Ker(\widetilde{e}_{WW^*}).$$ We have $S^2(W), S^2(W^*) \subseteq \Ker \varphi$ since $W$ and $W^*$ are totally singular. Furthermore, the smallest Jordan block size in $W \otimes W^*$ is $2^{\alpha}$ by Proposition \ref{prop:uninilsim} and \cite[Lemma 4.2]{KorhonenJordanGood}, so (i) follows from Lemma \ref{jordanrestrictionNIL}.

Next we consider (ii). We have $V \downarrow K[e^{2^{\alpha}}] = W(\ell/2^{\alpha})^{2^{\alpha}}$ by Lemma \ref{lemma:nilsquaredecomp}. Therefore by Lemma \ref{lemma:NILS2decomp}, it suffices to consider the case where $\alpha = 0$, in which case $\ell$ is odd. In this case, choose a basis $v_1$, $\ldots$, $v_{\ell}$ of $W$, and let $w_1$, $\ldots$, $w_{\ell}$ be the corresponding dual basis in $W^*$, so $b(v_i,w_j) = \delta_{i,j}$ for all $1 \leq i,j \leq \ell$. Then $\gamma = \sum_{1 \leq i \leq \ell} v_iw_i$ is annihilated by $e$ and $\varphi(\gamma) = \ell \neq 0$, so $\Ker(\widetilde{e}) \not\subseteq \Ker \varphi$, as required.\end{proof}

\begin{lemma}\label{lemma:S2indecompWlkNIL}
Let $e \in \mathfrak{sp}(V)$ be nilpotent such that $V \downarrow K[e] = W_k(\ell)$, where $0 < k < \ell/2$. Let $\widetilde{e}$ be the action of $e$ on $S^2(V)$, and denote $\alpha = \nu_2(\ell)$. Then the following hold:
	\begin{enumerate}[\normalfont (i)]
		\item $\Ker(\widetilde{e})^{2^{\alpha}-1} \subseteq \Ker \varphi$.
		\item $\Ker(\widetilde{e})^{2^{\alpha}} \not\subseteq \Ker \varphi$.
	\end{enumerate}
\end{lemma}

\begin{proof}Let $v_1$, $\ldots$, $v_n$ be the basis used in the definition of $W_k(\ell)$. Then $b(v_i,v_j) = 1$ if $i+j = n+1$ and $0$ otherwise. Furthermore, the action of $e$ is defined by \begin{align*}
ev_1 &= 0, \\
ev_{\ell+1} &= 0 \\
ev_i &= v_{i-1} \text{ for all } i \not\in \{1,\ell+1,n-k+1\} \\
ev_{n-k+1} &= v_{n-k} + v_{k}. \end{align*} We will denote $v_j = 0$ for all $j \leq 0$ and $j > n$.

Let $W = \langle v_i : 1 \leq i \leq \ell \rangle$ and $Z = \langle e^i v_n : 0 \leq i < \ell \rangle$. Then $V = W \oplus Z$, where $W$ and $Z$ are $K[e]$-modules on which $e$ acts as a single Jordan block of size $\ell$. We take $w_1$, $\ldots$, $w_{\ell}$ as a basis of $Z$, where $w_i := e^{i-1} v_n$ for all $1 \leq i \leq \ell$.

For the claims, we will first consider the case where $\alpha = 0$, so $\ell$ is odd. In this case (i) is trivial. For (ii), note that the vector $\gamma = \sum_{1 \leq j \leq \ell} v_j w_{j}$ is annihilated by the action of $e$. For all $1 \leq j \leq \ell$, we have $w_{j} = v_{n-j+1}$ or $w_{j} = v_{n-j+1} + v_r$ for some $1 \leq r \leq k$. Therefore $\varphi(\gamma) = \ell \neq 0$, so $\Ker(\widetilde{e}) \not\subseteq \Ker \varphi$, as claimed.

Suppose then that $\alpha > 0$. For (i), we will first prove that $\Ker(\widetilde{e}) \subseteq \Ker \varphi$. To this end, note that the decomposition $V = W \oplus Z$ induces a decomposition $$S^2(V) = S^2(W) \oplus S^2(Z) \oplus WZ$$ of $K[e]$-modules, where $WZ = \{ wz : w \in W, z \in Z\}$ is isomorphic to $W \otimes Z$. It follows from \cite[Proof of Theorem 1.6]{KorhonenSymExt2021} that $S^2(W)^e = W^{[2]}$ and $S^2(Z)^e = Z^{[2]}$, so $$\Ker(\widetilde{e}) = W^{[2]} \oplus Z^{[2]} \oplus \Ker(\widetilde{e}_{WZ}).$$ It is clear that $W^{[2]}, Z^{[2]} \subseteq \Ker \varphi$, and $\Ker(\widetilde{e}_{WZ}) \subseteq \Ker \varphi$ since the smallest Jordan block size in $W \otimes Z$ is $2^{\alpha}$ (Proposition \ref{prop:uninilsim} and \cite[Lemma 4.2]{KorhonenJordanGood}).

Thus $\Ker(\widetilde{e}) \subseteq \Ker \varphi$. This also proves (i) in the case where $\alpha = 1$, so suppose next that $\alpha > 1$. We have $V \downarrow K[e^{2^{\alpha-1}}] = W(\ell/2^{\alpha-1})^{2^{\alpha-1}}$ by Lemma \ref{lemma:nilsquaredecomp}. Since $\ell/2^{\alpha-1}$ is even, it follows from Lemma \ref{lemma:S2indecompWlNIL} and Lemma \ref{lemma:NILS2decomp} that \begin{equation}\label{eq:wklalpha1}\Ker(\widetilde{e})^{2^{\alpha-1}} \subseteq \Ker \varphi.\end{equation}

In $S^2(W)$, $S^2(Z)$, and $W \otimes Z$ all Jordan block sizes of $e$ are powers of two (Theorem \ref{thm:extsymnilpotent}). In particular $\widetilde{e}$ has no Jordan blocks of size $2^{\alpha-1} < d < 2^{\alpha}$, so by~\eqref{eq:wklalpha1} and Lemma \ref{jordanrestrictionNIL} we conclude that $\Ker(\widetilde{e})^{2^{\alpha}-1} \subseteq \Ker \varphi$, as required.

It remains to prove (ii) in the case where $\alpha > 0$. To this end, note first that we have $V \downarrow K[e^{2^{\alpha}}] = W(\ell/2^{\alpha})^{2^{\alpha}}$ by Lemma \ref{lemma:nilsquaredecomp}. It follows then from Lemma \ref{lemma:S2indecompWlNIL} and Lemma \ref{lemma:NILS2decomp} that $\Ker(\widetilde{e})^{2^{\alpha}} \not\subseteq \Ker \varphi$.\end{proof}

\begin{lemma}\label{lemma:S2indecompV2lNIL}
Let $e \in \mathfrak{sp}(V)$ be nilpotent such that $V \downarrow K[e] = V(2\ell)$. Let $\widetilde{e}$ be the action of $e$ on $S^2(V)$, and denote $\alpha = \nu_2(2\ell)$. Then the following hold:
	\begin{enumerate}[\normalfont (i)]
		\item $\Ker(\widetilde{e})^{2^{\alpha}-1} \subseteq \Ker \varphi$.
		\item $\Ker(\widetilde{e})^{2^{\alpha}} \not\subseteq \Ker \varphi$.
	\end{enumerate}
\end{lemma}

\begin{proof}
We have $\Ker(\widetilde{e}) = V^{[2]}$ by \cite[Proof of Theorem 1.6]{KorhonenSymExt2021}, so $\Ker(\widetilde{e}) \subseteq \Ker \varphi$. Since $e$ has no Jordan blocks of size $1 < d < 2^{\alpha}$ on $S^2(V)$ (Theorem \ref{thm:extsymnilpotent}), it follows from Lemma \ref{jordanrestrictionNIL} that $\Ker(\widetilde{e})^{2^{\alpha}-1} \subseteq \Ker \varphi$.

Next we consider claim (ii). Since $\alpha = \nu_2(2\ell) > 0$, we have $V \downarrow K[e^{2^{\alpha}}] = W(\ell/2^{\alpha-1})^{2^{\alpha-1}}$ by Lemma \ref{lemma:nilsquaredecomp}. Because $\ell/2^{\alpha-1}$ is odd, we conclude from Lemma \ref{lemma:S2indecompWlNIL} and Lemma \ref{lemma:NILS2decomp} that $\Ker(\widetilde{e})^{2^{\alpha}} \not\subseteq \Ker \varphi$.\end{proof}

\begin{prop}\label{prop:nilKeraction}
Let $e \in \mathfrak{sp}(V)$ be nilpotent, with orthogonal decomposition $$V \downarrow K[e] = \sum_{1 \leq i \leq t} W(m_i) \perp \sum_{1 \leq j \leq t'} W_{k_j}(\ell_j) \perp \sum_{1 \leq r \leq t''} V(2d_r).$$ Denote by $\alpha \geq 0$ the largest integer such that $2^{\alpha} \mid m_i,\ell_j,2d_r$ for all $i$, $j$, and $r$.

Then \begin{align*}
		\g_{sc} &\cong W_{2^{\alpha}} \oplus V' \\
		\g_{sc}/Z(\g_{sc}) &\cong W_{2^{\alpha}-1} \oplus V' \end{align*} for some $K[e]$-module $V'$.
\end{prop}

\begin{proof}Let $\widetilde{e}$ be the action of $e$ on $S^2(V)$. By Lemma \ref{lemma:LieSpS2}, Lemma \ref{lemma:kerphiisdual}, and Lemma \ref{jordanrestrictionNIL}, the proposition is equivalent to the statement that $\Ker(\widetilde{e})^{2^{\alpha}-1} \subseteq \Ker \varphi$ and $\Ker(\widetilde{e})^{2^{\alpha}} \not\subseteq \Ker \varphi$. 

Thus similarly to the proof of Proposition \ref{prop:unipKeraction}, the result follows from Lemma \ref{lemma:NILS2decomp}, together with Lemma \ref{lemma:S2indecompWlNIL}, Lemma \ref{lemma:S2indecompWlkNIL}, and Lemma \ref{lemma:S2indecompV2lNIL}.\end{proof}

\section{Jordan block sizes of unipotent elements on \texorpdfstring{$\g_{ad}$}{g\_ad}}\label{section:unipGAdZ}

Let $G_{sc} = \Sp(V)$ be simply connected and simple of type $C_{\ell}$. In this section, we will describe the Jordan block sizes of unipotent elements $u \in G_{sc}$ acting on $\g_{ad}$. Our approach is to describe the Jordan block sizes of $u$ on $\g_{ad}$ in terms of the Jordan block sizes on $[\g_{ad}, \g_{ad}]$, using Lemma \ref{jordanrestrictionNIL}. Since the Jordan block sizes of $u$ on $[\g_{ad}, \g_{ad}] \cong \g_{sc}/Z(\g_{sc})$ (Lemma \ref{lemma:gmodZisgadgad}) are known by the results of Section \ref{section:unipgmodZ}, we then get an explicit description of the Jordan block sizes of $u$ on $\g_{ad}$.

Throughout this section we will consider $\g_{ad}$ as constructed in Section \ref{subsection:chevalleyC} -- \ref{subsection:symsquareC}, and we use the notation established there.

\begin{lemma}\label{lemma:deltaWLunip}
Let $u \in \Sp(V)$ be the reduction modulo $p$ of $u_{\Z} = x_{\Z}^{(\alpha_1)} \cdots x_{\Z}^{(\alpha_{\ell-1})}$ as in Section \ref{subsection:unipotentroots}, so that $V \downarrow K[u] = W(\ell)$. Then the following hold:
	\begin{enumerate}[\normalfont (i)]
		\item $(u_{\Z} - 1) \cdot \delta = 0$.
		\item There exists $v \in L_{sc}$ such that $(u_{\Z} - 1) \cdot (\delta + v) \equiv 2 \delta \mod{2 L_{sc}}$ if and only if $\ell$ is even.
	\end{enumerate}
\end{lemma}

\begin{proof}
For claim (i), consider $X_{\alpha_i} = E_{i,i+1} - E_{n-i,n-i+1}$ for $1 \leq i \leq \ell-1$ from the Chevalley basis of $\mathfrak{sp}(V_{\C})$. It is clear that $X_{\alpha_i} \cdot \delta = 0$, so $x_{\Z}^{(\alpha_i)} \cdot \delta = \delta$ for all $1 \leq i \leq \ell-1$, from which (i) follows.

For (ii), we have $(u_{\Z} - 1) \cdot (\delta + v) = (u_{\Z} - 1) \cdot v$ for all $v \in L_{sc}$ by (i). By Lemma \ref{lemma:imZforUZ}, there exists $v \in L_{sc}$ such that $(u_{\Z} - 1) \cdot (v) \equiv 2 \delta \mod{2 L_{sc}}$ if and only if $\operatorname{Im}(\widehat{u} - 1) \supseteq Z(\g_{sc})$, where $\widehat{u}$ is the action of $u$ on $\g_{sc}$. It follows from Proposition \ref{prop:unipKeraction} and Lemma \ref{jordanquotientNIL} that $\operatorname{Im}(\widehat{u} - 1) \supseteq Z(\g_{sc})$ if and only if $\ell$ is even, so we conclude that (ii) holds.\end{proof}

\begin{lemma}\label{lemma:deltaV2Lunip}
Assume that $\ell$ is odd. Let $u \in \Sp(V)$ be the reduction modulo $p$ of $u_{\Z} = x_{\Z}^{(\alpha_1)} \cdots x_{\Z}^{(\alpha_{\ell})}$ as in Section \ref{subsection:unipotentroots}, so that $V \downarrow K[u] = V(2\ell)$. Then the following hold:

	\begin{enumerate}[\normalfont (i)]
		\item There exists $v \in L_{sc}$ such that $(u_{\Z} - 1) \cdot (\delta + v) \equiv 2 \delta \mod{2 L_{sc}}$.
		\item There does not exist $v \in L_{sc}$ such that $(u_{\Z} - 1) \cdot (\delta + v) \equiv 0 \mod{2 L_{sc}}$.
	\end{enumerate}

\end{lemma}

\begin{proof}
For (i), first we calculate \begin{align*} u_{\Z} \cdot \delta &= x_{\Z}^{(\alpha_1)} \cdots x_{\Z}^{(\alpha_{\ell-1})} \cdot ( \delta + \frac{1}{2} v_{\ell}^2 ) & (\text{by } x_{\Z}^{(\alpha_{\ell})} \cdot \delta = \delta + \frac{1}{2} v_{\ell}^2) \\ &= \delta + \frac{1}{2} \left(v_1 + \cdots + v_{\ell} \right)^2 & (\text{by Lemma \ref{lemma:deltaWLunip} (i)}) \\ &= \delta + \sum_{1 \leq i \leq \ell} \frac{1}{2} v_i^2 + \sum_{1 \leq i < j \leq \ell} v_iv_j.\end{align*}

Furthermore $$u_{\Z} \cdot \frac{1}{2} v_{\ell+1}^2 = \frac{1}{2} \left(v_1 + \cdots + v_{\ell+1} \right)^2 = \sum_{1 \leq i \leq \ell+1} \frac{1}{2}v_i^2 + \sum_{1 \leq i < j \leq \ell+1} v_iv_j,$$ so \begin{equation}\label{eq:firsteqV2ldelta}u_{\Z} \cdot \left(\delta + \frac{1}{2} v_{\ell+1}^2\right) \equiv \delta + \frac{1}{2} v_{\ell+1}^2 + \sum_{1 \leq i \leq \ell} v_iv_{\ell+1} \mod{2L_{sc}}.\end{equation} Next define $$\gamma := \sum_{1 \leq j \leq (\ell-1)/2} v_{2j}\left(v_{n-2j+1} + v_{n-2j+2}\right).$$ For $1 \leq j \leq (\ell-1)/2$, we denote $s_j := \sum_{1 \leq t \leq 2j} v_t v_{n-2j}$. Then modulo $2L_{sc}$, we have \begin{align*} u_{\Z} \cdot \gamma &\equiv \sum_{1 \leq j \leq (\ell-1)/2} \left( \sum_{1 \leq t \leq 2j} v_t \right) \left(v_{n-2j+2} + v_{n-2j} \right) \\ &\equiv \sum_{1 \leq j \leq (\ell-1)/2}\left( s_j + s_{j-1} + v_{2j-1}v_{n-2j+2} + v_{2j}v_{n-2j+2} \right) \\ &\equiv s_{(\ell-1)/2} + \sum_{1 \leq j \leq (\ell-1)/2} v_{2j-1} v_{n-2j+2} + \sum_{1 \leq j \leq (\ell-1)/2} v_{2j} v_{n-2j+2} \\ &\equiv s_{(\ell-1)/2} + \sum_{1 \leq j \leq (\ell-1)/2} v_{2j-1} v_{n-2j+2} + \left( \gamma + \sum_{1 \leq j \leq (\ell-1)/2} v_{2j} v_{n-2j+1} \right) \\ &\equiv \gamma + s_{(\ell-1)/2} + \sum_{1 \leq j \leq \ell-1} v_j v_{n-j+1} \\ &\equiv \gamma + s_{(\ell-1)/2} + \left(2\delta + v_{\ell}v_{\ell+1} \right) \\ &\equiv \gamma + 2 \delta + \sum_{1 \leq j \leq \ell} v_jv_{\ell+1}.\end{align*} Combining this with~\eqref{eq:firsteqV2ldelta}, we get $$u_{\Z} \cdot \left(\delta + \frac{1}{2} v_{\ell+1}^2 + \gamma\right) \equiv \delta + \frac{1}{2} v_{\ell+1}^2 + \gamma + 2 \delta \mod{2L_{sc}}.$$ We conclude then that (i) holds, with $v = \frac{1}{2} v_{\ell+1}^2 + \gamma$.

For claim (ii), suppose for the sake of contradiction that there exists $v \in L_{sc}$ such that $(u_{\Z} - 1) \cdot (\delta + v) \equiv 0 \mod{2 L_{sc}}$. Then it follows from (i) that there exists $w \in L_{sc}$ such that $(u_{\Z} - 1) \cdot w \equiv 2 \delta \mod{2 L_{sc}}$. By Lemma \ref{lemma:imZforUZ}, we have $\operatorname{Im}(\widehat{u} - 1) \supseteq Z(\g_{sc})$, where $\widehat{u}$ is the action of $u$ on $\g_{sc}$. On the other hand $\operatorname{Im}(\widehat{u} - 1) \not\supseteq Z(\g_{sc})$ by Proposition \ref{prop:unipKeraction} and Lemma \ref{jordanquotientNIL}, so we have a contradiction.\end{proof}

\begin{lemma}\label{lemma:GADoddsumofWs}
Let $u \in \Sp(V)$ be unipotent such that $V \downarrow K[u] = \sum_{1 \leq i \leq t} W(\ell_i)$, where $\ell_i \geq 1$ for all $1 \leq i \leq t$. Then $\Ker(\widetilde{u}-1) \not\subseteq [\g_{ad},\g_{ad}]$.
\end{lemma}

\begin{proof}
Follows from Lemma \ref{lemma:GADreduce} (i) and Lemma \ref{lemma:deltaWLunip} (i).
\end{proof}

\begin{lemma}\label{lemma:GADoddsumofVsandWs}
Let $u \in \Sp(V)$ be unipotent such that $V \downarrow K[u] = \sum_{1 \leq i \leq t} W(m_i) \perp \sum_{1 \leq j \leq s} V(2k_j)$, where $k_j \geq 1$ is odd for all $1 \leq j \leq s$. Assume that $s > 0$. 

Denote the action of $u$ on $\g_{ad}$ by $\widetilde{u}$. Then $\Ker(\widetilde{u}-1) \not\subseteq [\g_{ad},\g_{ad}]$ if and only if $m_i$ is even for all $1 \leq i \leq t$.
\end{lemma}

\begin{proof}
In view of Lemma \ref{lemma:deltaV2Lunip} (ii) and Lemma \ref{lemma:GADreduce} (i), we have $\Ker(\widetilde{u}-1) \not\subseteq [\g_{ad},\g_{ad}]$ if and only if Lemma \ref{lemma:GADreduce} (i)(b) holds. Then by Lemma \ref{lemma:deltaWLunip} (ii), we conclude that $\Ker(\widetilde{u}-1) \not\subseteq [\g_{ad},\g_{ad}]$ if and only if $m_i$ is even for all $1 \leq i \leq t$.\end{proof}

\begin{lemma}\label{lemma:GADindecompV2l}
Let $u \in \Sp(V)$ be unipotent such that $V \downarrow K[u] = V(2\ell)$. Let $\widetilde{u}$ be the action of $u$ on $\g_{ad}$, and denote $\alpha = \nu_2(\ell)$. Then the following hold:
	\begin{enumerate}[\normalfont (i)]
		\item $\Ker(\widetilde{u}-1)^{2^{\alpha}-1} \subseteq [\g_{ad}, \g_{ad}]$.
		\item $\Ker(\widetilde{u}-1)^{2^{\alpha}} \not\subseteq [\g_{ad}, \g_{ad}]$.
	\end{enumerate}
\end{lemma}

\begin{proof}We begin with the proof of (i), which is trivially true when $\alpha = 0$, so suppose that $\alpha > 0$.

We first consider $\alpha = 1$, in which case $\ell \equiv 2 \mod{4}$. Then \begin{align*} \dim \g_{ad}^u &\leq \dim \g_{sc}^u &&\text{by Lemma \ref{lemma:regularl2mod4ineq}} \\ &= \dim [\g_{ad},\g_{ad}]^u &&\text{by Lemma \ref{lemma:gmodZisgadgad} and Proposition \ref{prop:unipKeraction} } \\ &\leq \dim \g_{ad}^u. &&\end{align*} Thus we conclude that $\dim \g_{ad}^u = \dim \g_{sc}^u = \dim [\g_{ad},\g_{ad}]^u.$ It follows then from Lemma \ref{jordanrestrictionNIL} that $\Ker(\widetilde{u}-1) \subseteq [\g_{ad}, \g_{ad}]$, as claimed by (i).

Suppose then that $\alpha > 1$. We have $V \downarrow K[u^{2^{\alpha-1}}] = V(\ell/2^{\alpha-2})^{2^{\alpha-1}}$ by \cite[Lemma 6.13]{Korhonen2020Hesselink}. It follows then from Lemma \ref{lemma:GADreduce} (ii) and the case $\alpha = 1$ that $$\Ker(\widetilde{u}-1)^{2^{\alpha-1}} = \Ker(\widetilde{u}^{2^{\alpha-1}} - 1) \subseteq [\g_{ad}, \g_{ad}].$$ In particular, we have \begin{equation}\label{eq:kerisin}\Ker(\widetilde{u}-1) \subseteq [\g_{ad}, \g_{ad}].\end{equation} Next we note that it follows from Lemma \ref{lemma:gmodZisgadgad}, Proposition \ref{prop:unipKeraction}, and Lemma \ref{lemma:smallesblockunipS2} that the smallest Jordan block size of $u$ on $[\g_{ad}, \g_{ad}]$ is equal to $2^{\alpha}-1$. In particular, there are no Jordan blocks of size $1 \leq d < 2^{\alpha}-1$ for $u$ on $[\g_{ad}, \g_{ad}]$, which combined with Lemma \ref{jordanrestrictionNIL} and~\eqref{eq:kerisin} implies $\Ker(\widetilde{u}-1)^{2^{\alpha}-1} \subseteq [\g_{ad}, \g_{ad}]$.

For (ii), note that $V \downarrow K[u^{2^{\alpha}}] = V(\ell/2^{\alpha-1})^{2^{\alpha}}$ by \cite[Lemma 6.13]{Korhonen2020Hesselink}. Thus it follows from Lemma \ref{lemma:GADoddsumofVsandWs} that $\Ker(\widetilde{u}-1)^{2^{\alpha}} = \Ker(\widetilde{u}^{2^{\alpha}} - 1) \not\subseteq [\g_{ad}, \g_{ad}]$.\end{proof}

\begin{prop}\label{prop:unipgadaction}
Let $u \in \Sp(V)$ be unipotent, with orthogonal decomposition $V \downarrow K[u] = \sum_{1 \leq i \leq t} W(m_i) \perp \sum_{1 \leq j \leq s} V(2k_j)$. Assume that $s > 0$, and let $\alpha = \max_{1 \leq j \leq s} \nu_2(k_j)$. Let $\widetilde{u}$ be the action of $u$ on $\g_{ad}$. Then the following statements hold:

	\begin{enumerate}[\normalfont (i)]
		\item $\Ker(\widetilde{u} - 1)^{2^{\alpha}-1} \subseteq [\g_{ad}, \g_{ad}]$. 
		\item $\Ker(\widetilde{u} - 1)^{2^{\alpha}} \not\subseteq [\g_{ad}, \g_{ad}]$ if and only if $\alpha = \nu_2(k_j)$ for all $1 \leq j \leq s$, and $\nu_2(m_i) > \alpha$ for all $1 \leq i \leq t$.
		\item $\Ker(\widetilde{u} - 1)^{2^{\alpha}+1} \not\subseteq [\g_{ad}, \g_{ad}]$. 

	\end{enumerate}
\end{prop}

\begin{proof}
Claim (i) follows from Lemma \ref{lemma:GADreduce} and Lemma \ref{lemma:GADindecompV2l}.

Next we will prove (ii). In the case where $\alpha = 0$ we have $k_j$ odd for all $1 \leq j \leq s$. Then the claim of (ii) is that $\Ker(\widetilde{u} - 1) \not\subseteq [\g_{ad}, \g_{ad}]$ if and only if $m_i$ is even for all $1 \leq i \leq t$, which is precisely Lemma \ref{lemma:GADoddsumofVsandWs}.

Suppose then that $\alpha > 0$. Write $m_i = m_i' 2^{\alpha} + m_i''$ with $0 \leq m_i'' < 2^{\alpha}$ and $k_j = k_j' 2^{\alpha-1} + k_j''$ with $0 \leq k_j'' < 2^{\alpha-1}$. Then for the restrictions of the summands in $V \downarrow K[u^{2^{\alpha}}]$, we have \begin{align*}W(m_i) \downarrow K[u^{2^{\alpha}}] &= W(m_i')^{2^{\alpha}-m_i''} \perp W(m_i'+1)^{m_i''} \\ V(2k_j) \downarrow K[u^{2^{\alpha}}] &= \begin{cases} V(k_j/2^{\alpha-1})^{2^{\alpha}}, & \text{ if } \nu_2(k_j) = \alpha. \\ W(k_j')^{2^{\alpha-1}-k_j''} \perp W(k_j'+1)^{k_j''}, & \text{ if } \nu_2(k_j) < \alpha. \end{cases} \end{align*} for all $1 \leq i \leq t$ and $1 \leq j \leq s$, by \cite[Lemma 3.1, Lemma 6.12, Lemma 6.13]{Korhonen2020Hesselink}. In the above we denote $W(0) = 0$.

We have $\alpha = \nu_2(k_j)$ for some $j$, so $V(k_j/2^{\alpha-1})$ appears as a summand of $V \downarrow K[u^{2^{\alpha}}]$. Thus by Lemma \ref{lemma:GADoddsumofVsandWs}, we have $\Ker(\widetilde{u} - 1)^{2^{\alpha}} = \Ker(\widetilde{u}^{2^{\alpha}} - 1) \not\subseteq [\g_{ad}, \g_{ad}]$ if and only if all the summands of the form $W(m)$ in $V \downarrow K[u^{2^{\alpha}}]$ have $m$ even.\\

\noindent \emph{Claim 1: Assume that $\nu_2(m_i) \leq \alpha$ for some $i$. Then $\Ker(\widetilde{u} - 1)^{2^{\alpha}} \subseteq [\g_{ad}, \g_{ad}]$.} \\
\noindent It suffices to show that $W(m_i) \downarrow K[u^{2^{\alpha}}]$ has a summand $W(m)$ with $m$ odd. If $0 < m_i'' < 2^{\alpha}$, then $W(m_i) \downarrow K[u^{2^{\alpha}}]$ has $W(m_i')$ and $W(m_i'+1)$ as a summand. If $m_i'' = 0$, then $W(m_i) \downarrow K[u^{2^{\alpha}}] = W(m_i')^{2^{\alpha}}$ with $m_i'$ odd, since $\nu_2(m_i) \leq \alpha$.\\

\noindent \emph{Claim 2: Assume that $\nu_2(k_j) < \alpha$ for some $j$. Then $\Ker(\widetilde{u} - 1)^{2^{\alpha}} \subseteq [\g_{ad}, \g_{ad}]$.} \\
It suffices to show that $V(2k_j) \downarrow K[u^{2^{\alpha}}]$ has a summand $W(m)$ with $m$ odd. If $0 < k_j'' < 2^{\alpha-1}$, then $V(2k_j) \downarrow K[u^{2^{\alpha}}]$ has $W(k_j')$ and $W(k_j'+1)$ as a summand. Suppose then that $k_j'' = 0$. Then $V(2k_j) \downarrow K[u^{2^{\alpha}}] = W(k_j')^{2^{\alpha-1}}$, with $k_j'$ is odd since $\nu_2(k_j) < \alpha$.\\

The ``only if'' part of (ii) follows from Claim 1 and Claim 2. Conversely, if $\alpha = \nu_2(k_j)$ for all $1 \leq j \leq s$ and $\nu_2(m_i) > \alpha$ for all $1 \leq i \leq t$, we have $$V \downarrow K[u^{2^{\alpha}}] = \sum_{1 \leq i \leq t} W(m_i/2^{\alpha})^{2^{\alpha}} \perp \sum_{1 \leq j \leq s} V(k_j/2^{\alpha-1}).$$ Thus $\Ker(\widetilde{u} - 1)^{2^{\alpha}} = \Ker(\widetilde{u}^{2^{\alpha}} - 1) \not\subseteq [\g_{ad}, \g_{ad}]$ by Lemma \ref{lemma:GADoddsumofVsandWs}.

For (iii), let $u_{\Z} = u_{\Z,1} \cdots u_{\Z,t+s} \in \mathscr{U}_{\Z}$ be as in Section \ref{subsection:unipotentroots}, so that $u$ is the reduction modulo $p$ of the action of $u_{\Z}$ on $V_{\Z}$. Denote $\delta_1$, $\ldots$, $\delta_{t+s}$ as in Section \ref{subsection:unipotentroots}. It follows from Lemma \ref{lemma:GADreduce} (i), Lemma \ref{lemma:deltaWLunip}, and Lemma \ref{lemma:GADindecompV2l} that there exist $w_{i} \in L_{sc}^{(i)}$ such that \begin{align*}(u_{\Z,i}-1)^{2^{\alpha}} \cdot (\delta_i + w_i) &\equiv 0 \mod{2 L_{sc}^{(i)}} & \text{ if } 1 \leq i \leq t, \\ (u_{\Z,i}-1)^{2^{\alpha}} \cdot (\delta_i + w_i) &\equiv 2 \delta_{i} \mod{2 L_{sc}^{(i)}} & \text{ if } t+1 \leq i \leq t+s.\end{align*} Then $$(u_{\Z}-1)^{2^{\alpha}} \cdot (\delta + w_1 + \cdots + w_{t+s}) \equiv 2\delta_{t+1} + \cdots + 2\delta_{t+s} \mod{2L_{sc}}.$$ On the other hand, we have $u_{\Z,i} \cdot \delta_i = \delta_i + z_i$ for some $z_i \in L_{sc}^{(i)}$, as seen in the beginning of the proof of Lemma \ref{lemma:deltaV2Lunip}. Thus $(u_{\Z}-1) \cdot 2\delta_i \equiv 0 \mod{2L_{sc}^{(i)}}$ for all $t+1 \leq i \leq t+s$, and consequently $$(u_{\Z}-1)^{2^{\alpha}+1} \cdot (\delta + w_1 + \cdots + w_{t+s}) \equiv 0 \mod{2L_{sc}}.$$ It follows then from Lemma \ref{lemma:GADreduce} that $\Ker(\widetilde{u} - 1)^{2^{\alpha}+1} \not\subseteq [\g_{ad}, \g_{ad}]$, as claimed by (iii).\end{proof}

\begin{prop}\label{prop:unipGAD}
Let $u \in \Sp(V)$ be unipotent, with orthogonal decomposition $V \downarrow K[u] = \sum_{1 \leq i \leq t} W(m_i) \perp \sum_{1 \leq j \leq s} V(2k_j)$. Then the following statements hold:
	\begin{enumerate}[\normalfont (i)]
		\item Suppose that $s = 0$. Then $$\g_{ad} \cong V_1 \oplus [\g_{ad}, \g_{ad}]$$ as $K[u]$-modules.
		
		\item Suppose that $s > 0$, and let $\alpha = \max_{1 \leq j \leq s} \nu_2(k_j)$. Then:
			\begin{enumerate}[\normalfont (a)]
				\item If $\alpha = \nu_2(k_j)$ for all $1 \leq j \leq s$ and $\nu_2(m_i) > \alpha$ for all $1 \leq i \leq t$, then \begin{align*}
		\g_{ad} &\cong V_{2^{\alpha}} \oplus V' \\
		[\g_{ad},\g_{ad}] &\cong V_{2^{\alpha}-1} \oplus V' \end{align*} for some $K[u]$-module $V'$.
				
				\item If $\alpha > \nu_2(k_j)$ for some $1 \leq j \leq s$, or $\nu_2(m_i) \leq \alpha$ for some $1 \leq i \leq t$, then \begin{align*}
		\g_{ad} &\cong V_{2^{\alpha}+1} \oplus V' \\
		[\g_{ad},\g_{ad}] &\cong V_{2^{\alpha}} \oplus V' \end{align*} for some $K[u]$-module $V'$.
			\end{enumerate}
		
	\end{enumerate}

\end{prop}

\begin{proof}
Claim (i) follows from Lemma \ref{lemma:GADoddsumofWs} and Lemma \ref{jordanrestrictionNIL}. Similarly (ii) follows from Proposition \ref{prop:unipgadaction} and Lemma \ref{jordanrestrictionNIL}.
\end{proof}

\section{Jordan block sizes of nilpotent elements on \texorpdfstring{$\g_{ad}$}{g\_ad}}\label{section:nilGAdZ}

Continuing with the setup of the previous section, in this section we describe the Jordan block sizes of nilpotent $e \in \g_{sc}$ on $\g_{ad}$, in terms of the Jordan block sizes on $[\g_{ad},\g_{ad}]$. The basic approach is similar to the previous section, but the proofs will be more simple due to the fact that $V \downarrow K[e^2]$ is always has an orthogonal decomposition of the form $\sum_{1 \leq i \leq t} W(m_i)$ (Lemma \ref{lemma:nilsquaredecomp}).

\begin{lemma}\label{lemma:NILgadsingular}
Let $e \in \g_{sc}$ be nilpotent such that $V \downarrow K[e] = \sum_{1 \leq i \leq t} W(m_i)$, and let $\widetilde{e}$ be the action of $e$ on $\g_{ad}$. Then $\Ker \widetilde{e} \not\subseteq [\g_{ad}, \g_{ad}]$.
\end{lemma}

\begin{proof}
Let $e_{\Z} = e_{\Z,1} + \cdots + e_{\Z,t} \in \g_{\Z}$ be as in Section \ref{subsection:nilpotentroots}, so that $e$ is the reduction modulo $p$ of $e_{\Z}$. Then $$e_{\Z,i} = X_{\alpha_{t_i+1}} + \cdots + X_{\alpha_{t_i+m_i-1}},$$ where $t_1 = 0$ and $t_i = m_1 + \cdots + m_{i-1}$ for all $1 < i \leq t$. 

As noted earlier (proof of Lemma \ref{lemma:deltaWLunip}), we have $X_{\alpha_i} \cdot \delta = 0$ for all $1 \leq i < \ell$. Therefore $e_{\Z} \cdot \delta = 0$, from which it follows that $\Ker \widetilde{e} \not\subseteq [\g_{ad}, \g_{ad}]$ (Lemma \ref{lemma:nilpinUZgen}).
\end{proof}

\begin{lemma}\label{lemma:NILgadV2lWkl}
Let $e \in \g_{sc}$ be nilpotent with $V \downarrow K[e] = V(2\ell)$ or $V \downarrow K[e] = W_k(\ell)$ for some $0 < k < \ell/2$. Let $\widetilde{e}$ be the action of $e$ on $\g_{ad}$. Then $\Ker \widetilde{e} \subseteq [\g_{ad}, \g_{ad}]$.
\end{lemma}

\begin{proof}
Let $e_{\Z} \in \g_{\Z}$ be as in Section \ref{subsection:nilpotentroots}, so that $e$ is the reduction modulo $p$ of $e_{\Z}$. Then $$e_{\Z} = X_{\alpha_1} + \cdots + X_{\alpha_{\ell-1}} + X_{2 \varepsilon_r},$$ where $r = \ell$ if $V \downarrow K[e] = V(2\ell)$, and $r = k$ if $V \downarrow K[e] = W_k(\ell)$.

Suppose that $\Ker \widetilde{e} \not\subseteq [\g_{ad}, \g_{ad}]$. Then it follows from Lemma \ref{lemma:nilpinUZgen} that there exists $v \in L_{sc}$ such that $e_{\Z} \cdot (\delta + v) \in 2L_{ad}$. We have $e_{\Z} \cdot \delta = \frac{1}{2} v_r^2,$ so \begin{equation}\label{eq:12vr2}\frac{1}{2} v_r^2 = -e_{\Z} \cdot v + 2w\end{equation} for some $w \in L_{ad}$.

Note that $$e_{\Z} \cdot \frac{1}{2} v_i^2 = \pm v_i v_{i-1}$$ for all $1 \leq i \leq n$. Therefore $e_{\Z} L_{sc} \subseteq S^2(V_{\Z})$. Since also $2L_{ad} \subseteq S^2(V_{\Z})$, it follows from~\eqref{eq:12vr2} that $\frac{1}{2} v_r^2 \in S^2(V_{\Z})$, which is impossible. We have a contradiction, so we conclude that $\Ker \widetilde{e} \subseteq [\g_{ad}, \g_{ad}]$.\end{proof}

\begin{lemma}\label{lemma:nonsingularkernelNILGAD}
Let $e \in \g_{sc}$ be nilpotent, and let $\widetilde{e}$ be the action of $e$ on $\g_{ad}$. Then $\Ker \widetilde{e} \subseteq [\g_{ad}, \g_{ad}]$, except possibly when $V \downarrow K[e] = \sum_{1 \leq i \leq t} W(\ell_i)$ for some $\ell_1$, $\ldots$, $\ell_t$.\end{lemma}

\begin{proof}
Write $V \downarrow K[e] = U_1 \perp \cdots \perp U_t$, where $U_i$ is orthogonally indecomposable for all $1 \leq i \leq t$, and $\dim U_i = 2\ell_i$ with $\ell_i > 0$. Let $e_{\Z} = e_{\Z,1} + \cdots + e_{\Z,t} \in \g_{\Z}$ be as in Section \ref{subsection:nilpotentroots}, so that $e$ is the reduction modulo $p$ of $e_{\Z}$. 

The result follows from Lemma \ref{lemma:GADreduceNIL} (ii) (with $X_i = e_{\Z,i}$) and Lemma \ref{lemma:NILgadV2lWkl}.\end{proof}

\begin{lemma}\label{lemma:nonsingularkernel2NILGAD}
Let $e \in \g_{sc}$ be nilpotent, and let $\widetilde{e}$ be the action of $e$ on $\g_{ad}$. Then $\Ker (\widetilde{e})^2 \not\subseteq [\g_{ad}, \g_{ad}]$.
\end{lemma}

\begin{proof}
We have $V \downarrow K[e^2] = \sum_{1 \leq i \leq t} W(m_i)$ by Lemma \ref{lemma:nilsquaredecomp}, so the result follows from Lemma \ref{lemma:NILgadsingular}.
\end{proof}

\begin{prop}\label{prop:nilgadaction}
Let $e \in \g_{sc}$ be nilpotent. Then the following hold:

	\begin{enumerate}[\normalfont (i)]
		\item If $V \downarrow K[e] = \sum_{1 \leq i \leq t} W(m_i)$, then $$\g_{ad} \cong W_1 \oplus [\g_{ad}, \g_{ad}]$$ as $K[e]$-modules.

		\item If $V \downarrow K[e]$ is not of the form $\sum_{1 \leq i \leq t} W(m_i)$, then \begin{align*}
		\g_{ad} &\cong W_{2} \oplus V' \\
		[\g_{ad},\g_{ad}] &\cong W_{1} \oplus V' \end{align*} for some $K[e]$-module $V'$.
	\end{enumerate}
	
\end{prop}

\begin{proof}In case (i), the claim follows from Lemma \ref{lemma:NILgadsingular} and Lemma \ref{jordanrestrictionNIL}. Similarly (ii) follows from Lemma \ref{lemma:nonsingularkernelNILGAD}, Lemma \ref{lemma:nonsingularkernel2NILGAD}, and Lemma \ref{jordanrestrictionNIL}.\end{proof}

\section{Proofs of main results}\label{section:final}

We can now prove the main results stated in the introduction; all of them are straightforward consequences of the results from previous sections.

\begin{proof}[Proof of Theorem \ref{thm:unipGSCtoGAD}]
By Proposition \ref{prop:unipKeraction}, we have  \begin{equation}\label{eq:mainthmfinaleq}\begin{aligned}
		\g_{sc} &\cong V_{2^{\alpha}} \oplus V' \\
		\g_{sc}/Z(\g_{sc}) &\cong V_{2^{\alpha}-1} \oplus V' \end{aligned}\end{equation} for some $K[u]$-module $V'$.

If $s = 0$, then by Proposition \ref{prop:unipGAD} and Lemma \ref{lemma:gmodZisgadgad} we have $\g_{ad} \cong V_1 \oplus \g_{sc}/Z(\g_{sc}) \cong V_1 \oplus V_{2^{\alpha}-1} \oplus V'$ as $K[u]$-modules. Thus (i) holds.

For (ii), suppose that $s > 0$ and let $\beta = \max_{1 \leq j \leq s} \nu_2(k_j)$. Consider first the case where (ii)(a) holds, so $\nu_2(k_j) = \beta$ for all $1 \leq j \leq s$ and $\nu_2(m_i) > \beta$ for all $1 \leq i \leq t$. Then $\alpha = \beta$, so by Proposition \ref{prop:unipGAD} and Lemma \ref{lemma:gmodZisgadgad} we have \begin{align*}
		\g_{ad} &\cong V_{2^{\alpha}} \oplus V'' \\
		\g_{sc}/Z(\g_{sc}) &\cong V_{2^{\alpha}-1} \oplus V'' \end{align*} for some $K[u]$-module $V''$. We have $V' \cong V''$ by~\eqref{eq:mainthmfinaleq}, so $\g_{ad} \cong \g_{sc}$ as $K[u]$-modules.

We consider then the case where (ii)(b) holds, so either $\nu_2(k_j) < \beta$ for some $1 \leq j \leq s$, or $\nu_2(m_i) \leq \beta$ for some $1 \leq i \leq t$. Then by Proposition \ref{prop:unipGAD} and Lemma \ref{lemma:gmodZisgadgad}, we have \begin{align*}
		\g_{ad} &\cong V_{2^{\beta}+1} \oplus V'' \\
		\g_{sc}/Z(\g_{sc}) &\cong V_{2^{\beta}} \oplus V'' \end{align*} for some $K[u]$-module $V''$. By~\eqref{eq:mainthmfinaleq} we have $V'' \cong V_{2^{\alpha}-1} \oplus V'''$ and $V' \cong V_{2^{\beta}} \oplus V'''$ for some $K[u]$-module $V'''$. Thus \begin{align*}
		\g_{sc} &\cong V_{2^{\alpha}} \oplus V_{2^{\beta}} \oplus V''', \\
		\g_{ad} &\cong V_{2^{\alpha}-1} \oplus V_{2^{\beta}+1} \oplus V''' \end{align*} so claim (ii)(b) holds.\end{proof}

\begin{proof}[Proof of Theorem \ref{thm:nilGSCtoGAD}]
Similarly to Theorem \ref{thm:unipGSCtoGAD}, the result follows from Proposition \ref{prop:nilKeraction}, Lemma \ref{lemma:gmodZisgadgad}, and Proposition \ref{prop:nilgadaction}.
\end{proof}

\begin{proof}[Proof of Corollary \ref{cor:centralizerdimUNIP} and Corollary \ref{cor:centralizerdimNIL}]
Immediate from Theorem \ref{thm:unipGSCtoGAD} and Theorem \ref{thm:nilGSCtoGAD}.
\end{proof}

\begin{proof}[Proof of Corollary \ref{cor:regularcentralizerdim}]
For a regular unipotent $u \in \Sp(V)$ we have $V \downarrow K[u] = V(2\ell)$ \cite[p. 61]{LiebeckSeitzClass}. Then $\dim \g_{sc}^u = \ell+1$ by Lemma \ref{lemma:LieSpS2} and Lemma \ref{lemma:fixpdimsymwedge}. It follows then from Corollary \ref{cor:centralizerdimUNIP} that $\dim \g_{ad}^u = \dim \g_{sc}^u = \ell+1$.

For regular nilpotent $e \in \mathfrak{sp}(V)$ we have similarly $V \downarrow K[e] = V(2\ell)$ \cite[p. 60]{LiebeckSeitzClass}. It follows from Lemma \ref{lemma:LieSpS2} and \cite[Corollary 4.2]{KorhonenSymExt2021} that $\dim \g_{sc}^e = 2\ell$. Then $\dim \g_{ad}^e = \dim \g_{sc}^e = 2\ell$ by Corollary \ref{cor:centralizerdimNIL}, as claimed.
\end{proof}

\bibliographystyle{plain}
\bibliography{bibliography}

\end{document}